\documentclass[11pt,final]{article}
\usepackage{authbl
k}
\usepackage[letterpaper,margin=1in]{geometry}
\usepackage[utf8]{inputenc}
\usepackage[english]{babel}

\usepackage{epic,eepic}
\usepackage{mathtools}
\usepackage{braket}
\usepackage{amssymb,amsthm,amscd,latexsym,mathrsfs,units,enumerate,bm,bbm,cancel,physics,amsfonts,tikz, phonetic, multicol, float, caption, subcaption, comment}
\usepackage{color,xcolor}
\usepackage{graphicx,wrapfig}
\usepackage[font=footnotesize]{caption}

\usepackage[colorlinks]{hyperref}
\usepackage[notref, notcite]{showkeys}

\newcommand{\be}{\begin{equation}}
\newcommand{\eq}[2]{\begin{equation}\begin{split}#1\end{split}#2\end{equation}}
\newcommand{\ee}{\end{equation}}

\newcommand{\mat}[1]{\begin{pmatrix}#1\end{pmatrix}}
\newcommand{\imply}[1]{\Rightarrow}
\newcommand{\eps}{\varepsilon}

\newcommand{\umax}{\xi}

\newcommand{\cmt}[1]{}

\newtheorem{theorem}{Theorem}[section]

\newtheorem{lemma}[theorem]{Lemma}
\newtheorem{proposition}{Proposition}

\newtheorem{crit}[theorem]{Criterion}

\theoremstyle{definition}

\newtheorem{remark}{Remark}

\providecommand{\keywords}[1]{\textbf{Keywords:} #1}

\author[1]{Blake~Barker\thanks{E-mail:blake@mathematics.byu.edu}}
\author[2]{Emmanuel Fleurantin\thanks{E-mail:efleuran@gmu.edu}}
\author[3]{Matt Holzer\thanks{E-mail:mholzer@gmu.edu}}
\author[4]{Christopher K.R.T. Jones\thanks{E-mail:ckrtj@renci.org}}
\author[5]{Sebastian Wieczorek\thanks{E-mail:sebastian.wieczorek@ucc.ie}}
\affil[1]{\small Department of Mathematics, Brigham Young University, Provo, UT 84604, USA}
\affil[2,3,4]{\small{Department of Mathematics, George Mason University, Fairfax, VA 22030}}
\affil[5] {\small {School of Mathematical Sciences, University College Cork, Cork, Ireland, T12 XF62}}
\date{}                                           

\title{Rate-Induced Tipping in a Non-Uniformly Moving Habitat and Determination of the Critical Rate}

\begin{document}
\maketitle

\vspace{-3em}

\begin{abstract}
A habitat that is moving due to environmental change may result in tipping to extinction if the rate at which it moves is too great. We use a scalar reaction-diffusion equation with a non-autonomous reaction term, representing a spatially localized habitat moving from one asymptotic location to another, as a context for studying this phenomenon. The movement is characterized by displacement $d$ and rate parameter $r$. 
The system admits three steady states in both asymptotic habitat locations: a stable extinction state $u_0^*=0$, an unstable pulse (so-called edge state) $u_1^*(x)>0$, which gives rise to the Allee effect, and a stable pulse (populated base state) $u_2^*(x)>u_1^*(x)$, which corresponds to a thriving population at its carrying capacity. Numerical simulations for a specific model identify a critical displacement $d^*$ and, for $d > d^*$, demonstrate the existence of a \textit{critical rate} $r_c(d)$ at which rate-induced tipping occurs: for $r> r_c$ an initially thriving population becomes extinct due to habitat movement being too rapid. We provide analytical results for two limiting cases. For $r\ll 1$, solutions track the moving base state with error $O(r)$. For $r\gg 1$, solutions converge to the extinction state provided $d$ is sufficiently large. For $d$ too small, no tipping occurs regardless of $r$. Numerical simulations complement and extend these analytical results. At the critical rate $r=r_c(d)$, we identify a pulse-to-pulse heteroclinic connection between the base state at the past asymptotic location and the  edge state at the future asymptotic location of the habitat. 
We also establish the uniqueness of this critical rate and non-degeneracy of the heteroclinic connection as $r$ varies. 

\end{abstract}

\begin{flushleft}
\\

\vspace{1em}

\end{flushleft}

\section{Introduction}

\begin{figure}[t]
 \begin{center}
\includegraphics[scale=0.4]{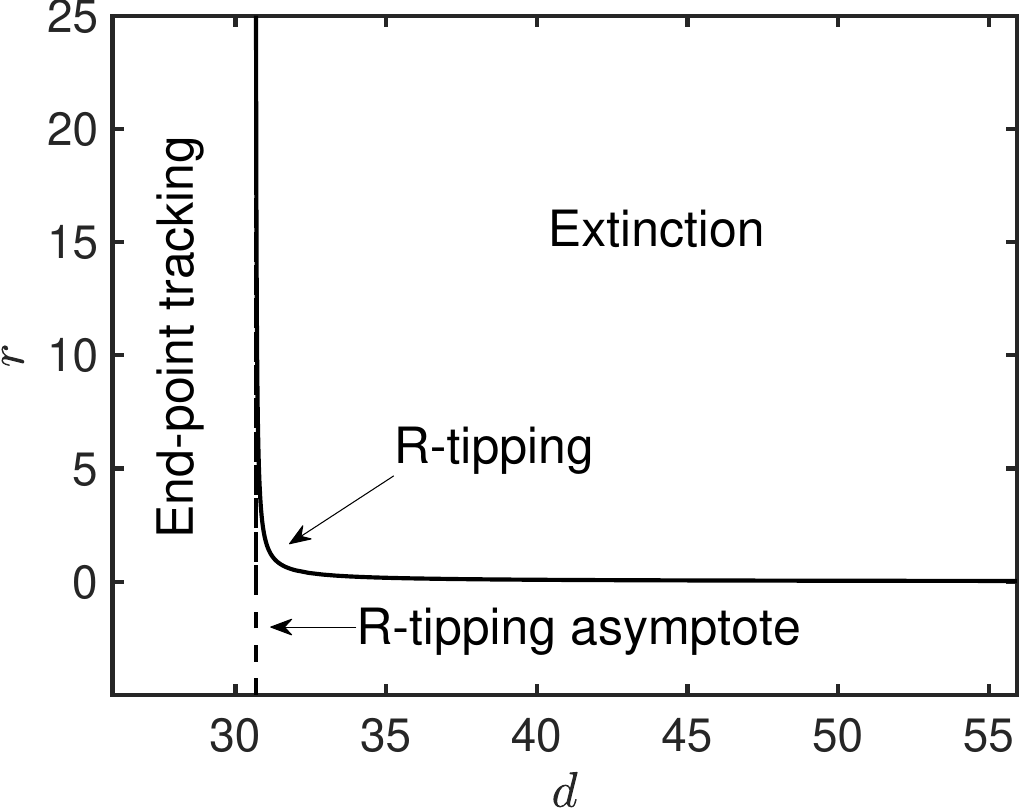}
\end{center}
\caption{Rate-induced tipping (R-tipping) diagram for the 
reaction-diffusion model~\eqref{eq:main} with $f$ defined as in~\eqref{eq:rterm} with parameters  $\beta = 0.15$, $\lambda = 4 \beta$ and $L = 25$, showing the relationship between habitat displacement $d$ and rate $r$ of habitat shift. 
The R-tipping curve separates parameter regions where the species persists (endpoint tracking) from the extinction region. For small displacements ($d  \lesssim 30$), no extinction occurs regardless of the rate, while larger displacements lead to extinction above some critical rate.
}
 \label{fig:criticalvalueplot}
\end{figure}

The movement of a habitat location could be forced by temperature or precipitation changes, such as occur in a warming climate. The habitat might move in a northerly direction (in the Northern Hemisphere) or up a mountain slope \cite{chen2011rapid,parmesan2003globally}. The issue we address here is how the rate at which such a habitat moves can result in the species going extinct, even though the habitat itself remains an environment conducive to the survival of the species in the new location.

This is naturally cast in terms of rate-induced tipping~\cite{ashwin2017parameter,wieczorek2023rate, ritchie2023rate}.  Tipping occurs in this case when external conditions change faster than some critical rate
and cause a transition away from a seemingly stable state, whereas a change in external conditions of the same magnitude, but at a slower rate, does not trigger such a transition. 
A major concern with climate change is that the underlying conditions are changing at a much faster rate than ever occurred in our planet's past \cite{diffenbaugh2013changes,loarie2009velocity}. This fact makes rate-induced tipping highly relevant to climate-driven systems. 

The scenario we address here is when a habitat shifts from one place to another. We imagine it to have been stably situated at a 
certain location and its transitioning to another location where it will stably sit under a new set of external conditions. The implications of a moving habitat have been addressed in previous work, starting with the work of Berestycki et al. for a mono-stable habitat, see \cite{berestycki2009can}. The case we consider is built on that formulated by Hasan et al. for a bi-stable habitat, see \cite{hasan2023rate}. In both of these works, the authors assume the habitat is moving steadily, which is represented by a linear ramp in the habitat position. Their work therefore does not take into account the acceleration and deceleration phases of the habitat shift.
While the steadily moving phase of the habitat shift is key, a full understanding of the dynamics caused by a shifting habitat will require accounting for variations in the speed of the shift. Our formulation thus involves a nonlinear ramp in the habitat position from one asymptotically static location to another.  

The casting of the habitat shift problem in this context introduces both new phenomena and new challenges. It is intuitively clear that a habitat will need to shift a certain distance to cause a problem for the survival of a species. We would expect that a species would survive a  location change that is small relative to the size of their habitat, however rapidly it might happen. We show that there is indeed a displacement threshold for rate-induced tipping to occur. This is illustrated in Figure 1, where the displacement of the habitat is represented by $d$, and the rate at which it moves by $r$. The displacement threshold line is just to the right of $d=30$.

With a steady shift, the problem can be cast as an ordinary differential equation (ODE), and the techniques of low-dimensional dynamical systems are invoked (\cite{berestycki2009can} and \cite{hasan2023rate}). Such a simplification is not available however, when the shift is time-varying. The problem has then to be analyzed as a full partial differential equation (PDE). In the work of \cite{ashwin2017parameter}, the characterization of rate-induced tipping in terms of a pullback attractor is first introduced, and we build on that approach here. We show there is a pullback attractor for the time-varying shift system. We derive the pullback attractor from a trajectory of a semiflow on an infinite-dimensional space.

Tipping in this ecological context means the population goes extinct due to its inability to keep up with the moving habitat. A key aspect of a rate-induced tipping analysis is the determination of a \textit{critical rate} that separates tracking, where the solution stays close to the moving base state, from tipping to the extinction state. In \cite{ashwin2017parameter}, it is shown that, at a critical rate, the pullback attractor asymptotes to an edge state in forward time, which is typically a saddle state with one unstable direction. We adopt a compactification framework~\cite{wieczorek2021} in which there is then a heteroclinic connection between the base state at the past asymptotic location of the habitat, from which the tipping occurs, to the edge  state at the future asymptotic location of the habitat, at the critical rate.  This pulse-to-pulse heteroclinic connection, which for us lies in an infinite-dimensional space, is not to be confused with the point-to-point heteroclinic connections that are the traveling pulses themselves~\cite{hasan2023rate}, and are trajectories for an ODE.

This paper is organized as follows. Section \ref{sec:pbsetting} establishes the mathematical framework, including the reaction-diffusion equation with moving habitat, the compactified system, key assumptions, and states our main Theorem: under fairly general assumptions, if the displacement $d$  is too small then rate-induced tipping does not occur, while we prove it does occur for large enough displacements. In Section \ref{sec:rigr}, a proof is given that solutions track the moving habitat at small rates ($r \ll 1$) ensuring survival, while large rates ($r \gg 1$) lead to extinction. This suggests a precise displacement threshold. While we do not prove that in general, the displacement threshold is seen to be unique for a specific model from \cite{hasan2023rate}. Moreover, the critical rate when the displacement is above that threhsold is also shown to be unique for this example. Section \ref{sec:comps} presents numerical methods for computing the critical rate $r_c(d)$ via pulse-to-pulse heteroclinic connections for the specific model. Section \ref{sec:conclusion} summarizes our findings and discusses implications for species survival under rapid environmental change.

\section{Mathematical framework}\label{sec:pbsetting}
\subsection{General set-up}

The species dynamics is assumed to be governed by a reaction-diffusion equation in one space dimension $x$. The habitat of the species  is denoted by $H(x)$, which is assumed to be non-negative.  The nonlinear reaction term $f(u, H(x))$ is dependent on $H(x)$, so that the dynamics are different inside the habitat, where the species can flourish, from outside, where it will not survive. We make the general assumptions that the function $f$ is $C^2$ on $\mathbb{R} \times \mathbb{R}$, and that $f(0,H)=0$ for all $H\in \mathbb{R}$. The function $H(x)$ itself is assumed to be $C^\infty$ and decaying exponentially as $x\to\pm\infty$.
The habitat $H(x)$ will move in time according to a function that centers it at a specific location. We denote this function $\gamma(t)$.

Denoting the species population density by $u$, the governing equation is
\begin{equation}\label{eq:main}
    u_t=u_{xx}+f\left(u,H(x-\gamma(t))\right).
\end{equation}
The diffusion term models migration, and the diffusion coefficient is set equal to $1$ for convenience. 
Note that the PDE~\eqref{eq:main} is nonautonomous in both $x$ and $t$. It can naturally be written in a moving frame with $z=x-\gamma(t)$. With independent variable $z$, then \eqref{eq:main} becomes
\begin{equation}\label{eq:main_mov}
     u_t=u_{zz}+\gamma_t(t)\, u_z+f\left(u,H(z)\right).
\end{equation}
We will actually specify the velocity of the shift function $\gamma(t)$ and introduce a rate parameter $r>0$ explicitly.
We take the shift to be a solution of an autonomous differential equation itself: $\gamma_t=r\,g(\gamma) $. Appending the equation for $\gamma$, we can then write the governing system as
\begin{equation}\label{main_mov}
\begin{split}
u_t=& u_{zz}+r\,g(\gamma)\,u_z+f\left(u,H(z)\right),\\
   \gamma_t=& r\,g(\gamma),
\end{split}
\end{equation}
which is non-autonomous in $z$, but no longer in $t$.
The fixed habitat positions are taken to be at $x=\pm a$, so that $d=2a$. To reflect the fact the habitat shifts from $-a$ (past asymptotic locsation) to $+a$ (future asymptotic location), we require that $g(\pm a)=0$ and $g(\gamma)>0$ if $\gamma \in (-a,+a)$. Unless explicitly stated otherwise, we will assume that  $\gamma(t)$ is parameterized so that $\gamma(0)=0$. 

We will work with a (semi-)dynamical system determined by \eqref{main_mov}, which we somewhat favor in this context as it is autonomous with regard to time. The augmentation of \eqref{eq:main_mov}, leading to \eqref{main_mov}, is akin to the standard addition of $t$ as a dependent variable that makes a nonautonomous system into an autonomous one.  But the use of $\gamma$ as opposed to $t$ as the extra variable is advantageous as $\gamma$ ranges over a compact domain. It is a special case of compactification for asymptotically autonomous systems~\cite{wieczorek2021}.
\begin{lemma}
    Equation \eqref{main_mov} generates a semiflow (dynamical system in $t\ge 0$) on $H^1\left(\mathbb{R}\right)\times[-a,+a]$.
\end{lemma}
\begin{proof}
    The operator $\frac{d^2}{dz^2}$ generates a $C^0$ semigroup on $H^1(\mathbb{R})$, see for instance \cite{pazy1983semigroups}. Since translation is continuous on $H^1(\mathbb{R})$, $L=\frac{d^2}{dz^2}+rg(\gamma)\frac{d}{dz}$ also generates a $C^0$ semigroup. 

    Using the smoothness of $f(u,H)$ in its arguments and the fact that $H^1(\mathbb{R})$ embeds continuously into $L^\infty (\mathbb{R})$, it can be shown that $f(*,H(z))$ is locally Lipschitz on $H^1(\mathbb{R})$. From Theorem 1.4 of Pazy \cite{pazy1983semigroups}, there will be mild solutions on a (local) time interval. The full requirements for a semiflow then follow from applying the usual Gronwall estimates to the integral equation representation, see \cite{batesjones}.
\end{proof}

  \begin{remark}
      There are two different approaches to this semigroup formulation. The operators in question actually generate something stronger than a $C^0$ semigroup, namely an analytic semigroup. Henry \cite{henry1981geometric} uses this formulation, and it has the advantage of giving regularity of solutions as part of the theory. It would be based on the space $X=L^2(\mathbb{R})$, and $H^1(\mathbb{R})$ appears as the fractional power space $X^\frac12$. The semiflow properties are then shown to  hold on $X^\frac12=H^1(\mathbb{R})$, see \cite{henry1981geometric}. We favor the $C^0$ semigroup approach here as the invariant manifold theorems are much clearer in this context, see \cite{batesjones}.
  \end{remark}
 
 The limiting cases of $\gamma=\pm a$ are both invariant and carry the semiflow of the reaction-diffusion equation with static habitat:
\begin{equation}\label{main_steady}
u_t=u_{zz}+f\left(u,H(z)\right).
\end{equation}
We consider the Allee effect case in which \eqref{main_steady} is a bistable system with a stable pulse $u^*_2(z) > 0$, a stable trivial solution $u^*_0(z)=0$, and an intermediate unstable pulse $0 < u^*_1(z) < u^*_2(z)$. The unstable pulse $u^*_1(z)$ is often known as the threshold state, but in the R-tipping literature has come to be known as an edge state~\cite{wieczorek2023rate}, and we will use that terminology here. 
Precise hypotheses are given below in Section \ref{sec:hyp}. 

Using this framework, we can give clear definitions and criteria for rate-induced tipping. The compactification allows us to formulate the main criterion around the behavior of a trajectory that is asymptotic, in backward time, to the stable base state of the past asymptotic habitat's location. This trajectory is a (local) pullback attractor in the sense of \cite{kloeden2011nonautonomous}. The pullback attractor is a curve of functions of $z$, parameterized by $t$, in $H^1(\mathbf{R})$.
Note that the central role played by the compactified system \eqref{main_mov} comes out in the proof of this Lemma.

\begin{lemma}\label{prop:pullback}Assume the hypotheses (H1-5) of Section \ref{sec:hyp}. 
For each $r>0$, there is a (local) pullback attractor of the nonautonomous equation \eqref{eq:main} 
\begin{equation}\label{eq:pullback}
    \mathcal{U}=\{ U(t)\in H^1(\mathbf{R}): U(t)=u^*(z,t), -\infty<t<+\infty \},
\end{equation}
that tends to $u_2^*$ as $t\to-\infty$. In particular, $u^*(z,t)$ satisfies the first equation of \eqref{main_mov}. Equivalently, $u^*(x-\gamma(t),t)$ satisfies \eqref{eq:main}.
\end{lemma}
 
\begin{proof}
Linearizing (\ref{main_mov}) at the steady state $(u^*_2,-a)$ we obtain the system
\begin{align}
    u_t &= u_{zz}+f_u(u^*_2,H(z))u+rg'(-a) \left(\partial_z u^*_2\right)\gamma,  \nonumber \\
    \gamma_t &= rg'(-a)\gamma.\label{eq:linatminusinfinity}
\end{align}
This system is in skew-product form, and the linearization is diagonal, which implies that the spectrum is the union of the strictly stable spectrum associated with the stable steady state $u^*_2$ and the unstable eigenvalue $rg'(-a)$ coming from the assumptions placed upon the shift velocity function $g(\gamma)$, see (H5) below. 

It follows that for any $r>0$, the steady state $(u_2^*,-a)$ of Equation~\eqref{main_mov} has a (local) one-dimensional unstable manifold $W^u_{\rm loc}(u_2^*,-a)$, albeit lying in an ambient infinite-dimensional space: 
 \[ W^u_{\rm loc}(u_2^*,-a)= \left\{ (u,\gamma) \ | \
u=\phi(\gamma;r), \ |\gamma+a|<\delta \right\}, \]
for some $\delta>0$, and the function $\phi$ satisfies
\begin{equation} \|\phi(\gamma;r)-u^*_2\|_{H^1}\leq C r (\gamma+a), \label{eq:CMestimate} \end{equation}
for some $C>0$ and $|\gamma+a|<\delta$. For any closed interval of non-negative $r$ values the constants $C$ and $\delta$ may be chosen independently of $r$, see \cite{batesjones}. 

Since $W^u_{\rm loc}(u_2^*,-a)$ is one-dimensional, the branch of the local unstable manifold with $\gamma>-a$, as given here, is a trajectory of \eqref{main_mov}. As such, it can be propagated in forward time under the semiflow of \eqref{main_mov} and the projection, with appropriate parameterization,  onto the first variable $u\in H^1(\mathbf{R})$ is the desired pullback attractor. The fact that it is a (local) pullback attractor in the sense of \cite{kloeden2011nonautonomous} follows from its being a part of $ W^u(u_2^*,-a)$. Its uniqueness follows easily from the one-dimensionality of $ W^u(u_2^*,-a)$.
\end{proof}

 The construction in the proof of Lemma \ref{prop:pullback} mirrors that in \cite{wieczorek2023rate} but in the context of PDEs. The pullback attractor can be used to give a direct condition for rate-induced tipping.

\begin{crit}\label{criterion}
{\bf \textit{Rate-induced tipping}} (R-tipping) will occur if the pullback attractor, given by Lemma \ref{prop:pullback}, $\mathcal{U}=\{U(t): -\infty<t<+\infty\}$, satisfies both:
\begin{enumerate}
    \item[(c1)] for sufficiently large $r$, $U(t) \to u_0^*=0$, as $t\to + \infty$, and,
    \item[(c2)] for sufficiently small $r$, $U (t)\to u_2^*$, as $t\to + \infty$.
\end{enumerate} 
\end{crit} 
The former situation, described by (c1) in Criterion \ref{criterion}, is referred to simply as \textit{tipping}, while we use the term \textit{end-point tracking} for (c2) as the pullback attractor approaches the end-point $u_2^*$ of the moving base state~\cite{wieczorek2023rate}. A consequence of R-tipping occurring in this sense is that, for initial states close to the base state $u^*_2$ at large negative time, we will have tipping to extinction from these states if $r$ is sufficiently large, and end-point tracking of the moving base state if $r$ is sufficiently small. This validates that Criterion \ref{criterion} implies that irreversible R-tipping occurs in the sense defined in Wieczorek et al.~\cite{wieczorek2023rate}. 

 \begin{remark}
     We will flip back and forth between the nonautonomous equation \eqref{eq:main_mov} and its autonomous counterpart \eqref{main_mov}. The move from \eqref{main_mov} to \eqref{eq:main_mov} is achieved by simply projecting onto the first variable, but care needs to be taken in ensuring that the parametrization of $\gamma(t)$ is chosen appropriately, as indicated in the proof of Lemma \ref{prop:pullback}. From the proof of Lemma \ref{prop:pullback}, we see that the pullback attractor $\mathcal{U}$ and the unstable manifold $ W^u(u_2^*,-a)$ are essentially the same, with the distinction being that $\mathcal{U}$ is a trajectory of \eqref{eq:main_mov}, and $ W^u(u_2^*,-a)$ is an invariant manifold, but also a trajectory, of \eqref{main_mov}.
 \end{remark}

\subsection{Main Theorem} \label{main_thm}

Under fairly general assumptions about the habitat and shift functions, we prove that rate-induced tipping occurs, in the sense of Criterion \ref{criterion}, provided the displacement is large enough. Recall that $d=2a$. 

\begin{theorem}\label{prop:main}
There exists a displacement $d^*$ such that for any displacement $d>d^*$,  R-tipping occurs.
\end{theorem}

The picture we believe holds in some generality is that the  pullback attractor, in forward time,  jumps from the 
base state $u^*_2(z)$ down to the extinction state $u^*_0(z)$ as $r$ increases. There will be a critical value of the rate $r=r_c$ where we expect the pullback attractor will asymptote in forward time to neither the extinction nor the 
base state. The expectation is that at $r=r_c$ the pullback attractor will tend to the unstable edge state, i.e.,
\[\mathcal{U}(t)\to u^*_1(z) \hspace{2mm}{\rm as } \hspace{2mm}t \to +\infty.\]
Such an orbit would be derived from 
an orbit of \eqref{eq:main_mov} that approaches the stable (base) state as $t\to-\infty$ and the unstable (edge) state as $t \to +\infty$.
Establishing the existence of such an orbit 
is not a trivial exercise, especially as the phase space it would live in is infinite-dimensional. 
Another question is whether the critical rate is unique. In that case, there would be an isolated rate 
separating extinction from survival. 
Determining conditions under which this picture of a single transition from survival to extinction with a unique critical rate holds is an interesting challenge. 

For a fuller picture of the geometry, we will consider the compactified system~\eqref{main_mov} and the stable manifold of the edge state $(u^*_1(z),a)$ at $\gamma = a$. From \cite{batesjones}, this is an infinite-dimensional manifold in $H^1(\mathbb{R})\times [-a,+a]$, and has codimension one. It is naturally viewed as the {\it{threshold}} for the R-tipping event~\cite{wieczorek2023rate}, and it locally divides the phase space into two parts, corresponding to initial states that lead to extinction versus survival. 
%
As $r$ passes through $r_c$, the positions of both the pullback attractor and the threshold change so that they cross.
A natural question is whether the crossing is non-degenerate.
Since the pullback attractor is 
contained in the unstable manifold of $(u^*_2(z),-a)$, this is a crossing of the unstable manifold of one 
steady-state solution with the stable manifold of another, leading to a true heteroclinic connection in the compactified phase space.
The non-degeneracy translates into this being a transverse crossing in the parameter $r$.

\subsection{Specific model}
\label{sec:spec_mod}

We carefully work out the details of this picture for a specific nonlinearity $f$, habitat $H$, and shift velocity field $g$.  The explicit form of this nonlinear reaction term $f$ is given by
\begin{equation}\label{eq:rterm}
    f(u, H(x)) = -\beta^2 u + \lambda H(x)u^2 -  u^3,
    \end{equation}
where $\beta \geq 0$ and are $\lambda \geq 0$ are parameters. We set $\lambda = 4\beta$, which corresponds to the rescaled system described in \cite{hasan2023rate} that captures the Allee effect (bistability) within a habitat. The unscaled habitat function in \cite{hasan2023rate} has a parameter $a >0$ that measures how sharply it transitions
from unfavorable to favorable habitat, while $L > 0$ estimates the spatial extent of the favorable habitat when transitions are sharp enough. In this work, we use the re-scaled habitat function from \cite{hasan2023rate} that depends on $L$ alone:   \begin{equation}\label{eq:habitatf}
       H(x) =  \frac{\tanh(x+\frac{L}{2}) - \tanh(x-\frac{L}{2})}{2 \tanh (\frac{ L}{2})}.
    \end{equation} 
The verification of the hypotheses (H1-4), see Section~\ref{sec:hyp} below, for this specific model is given in Section~\ref{sec:transverse}; 
note that in \cite{hasan2023rate},
the authors effectively verify the analytical hypotheses  through a mixture of analysis and computation. 
We also show in Section~\ref{sec:transverse} that a critical-rate heteroclinic-orbit exists through careful computations. We further verify that the critical rate is unique and that the switching of the pullback attractor from survival to extinction happens non-degenerately.

The main Theorem for the general problem shows there is a critical rate $r_c$ and that it depends on the habitat displacement $d=2a$. We do not know if the hypotheses (H1-5) we set on the nonlinear term and the shift function  are sufficient to guarantee that the critical rate is unique for the general problem. It is quite possible that 
there are specific models that switch back and forth between tipping and tracking more than once as $r$ varies.

\section{Analytical results: tracking and tipping}\label{sec:rigr}

\subsection{Hypotheses}\label{sec:hyp}

   We return now to the general context and shall assume that the habitat function $H(x)$ has a local maximum at $x=0$, and  decreases monotonically as $|x|\to \infty$.

   The key hypotheses are for the reaction-diffusion equation with static habitat
   \begin{equation}\label{eq:steady_RDE}u_t=u_{xx}+f(u,H(x)).\end{equation}
  
   \begin{itemize}
   \item[(H1)] Equation \eqref{eq:steady_RDE} has exactly three steady-state (time independent) solutions: the trivial steady state $u_0^*=0$ corresponding to extinction, and two other steady-states, both of which are positive and are denoted $u_1^*(x)$ and $u_2^*(x)$, where $u_2^*(x)>u_1^*(x)$ for all $x$. Both $u_1^*(x)$ and $u_2^*(x)$ are assumed to be smooth and decay exponentially as $x\to\pm\infty$.
   
   \item[(H2)] The spectrum of the linearization of \eqref{eq:steady_RDE} at $u_0^*=0$ lies on the negative real axis.
   
   \item[(H3)] The linearization of \eqref{eq:steady_RDE} at the steady-state (pulse) $u_1^*(x)>0$ has one positive eigenvalue with the remainder of the spectrum lying on the negative real axis.
   
   \item[(H4)] The linearization of \eqref{eq:steady_RDE} at the steady-state (pulse) $u_2^*(x)>0$ has its spectrum lying on the negative real axis.
   \end{itemize}

   Note that the spectrum of the linearization of \eqref{eq:steady_RDE}, at any steady-state, is real as the linearized operator is self-adjoint. Each of the steady-states have a physical meaning in the static habitat problem: $u_0^*=0$ is the extinction state, $u_1^*(x)$ is the 
   edge state, and $u_2^*(x)$ represents a population that is surviving at its carrying capacity and stable. The codimension-one stable manifold of the edge state $u_1^*(x)$ is the Allee threshold that separates survival from extinction.

   We also make an assumption about the shift velocity field.
   
\vspace{4mm}

   \noindent (H5) The shift velocity function $g(\gamma)$ is smooth and satisfies 
   \begin{align} &g(\pm a)=0, \ g'(\pm a)\neq 0 \nonumber,\;\;\mbox{and}\;\; g(\gamma)>0\;\;\mbox{for}\;\; \gamma\in(-a,a). 
       \end{align}

\subsection{Tracking at low rate}

Our first result shows that -- for a fixed shift velocity function $g(\gamma)$ --  tipping does not occur for any rate $r$ that is sufficiently small. 

\begin{proposition}{\bf [Tracking for $r\ll 1$]} \label{prop:tippingno}
    Fix the shift velocity function $g(\gamma)$, satisfying (H5).  Then there exists a $r_0>0$ such that for any $0<r<r_0$ the  pullback attractor $\mathcal{U}(t)$ converges to the stable pulse $u_2^*$ as $t\to+\infty$.  Moreover, there exists a $C>0$ such that 
    \begin{equation}
    \label{eq:convergencersmall} \mathrm{sup}_{z\in\mathbb{R}} |\mathcal{U}(t)-u_2^*(z)|< Cr,
    \end{equation}
    for any $t$.  In other words, tipping does not occur when $r$ is sufficiently small and the solution tracks the manifold of frozen stable-pulse profiles.  
\end{proposition}

\begin{proof}

Let $u(z,t)=u_2^*(z)+p(z,t)$ and obtain the following PDE for the perturbation $p(z,t)$:
\begin{equation}   \label{eq:pevolve} 
p_t=p_{zz}+rg(\gamma(t))p_z+f_u(u^*_2(z),H(z))p+rg(\gamma)\frac{\partial}{\partial z} \left(u^*_2(z)\right) +N(z,p),\end{equation}
where $N$ represents nonlinear terms in $p$.

Consider $0<r<r_0$ for some $r_0$ sufficiently small.  Then there exists a $T_0(r)$ for which $\gamma(T_0)=\gamma_0$ for some fixed $-a<\gamma_0<-a+\delta$.  We consider initial data $p(z,T_0)=\phi(z,\gamma_0;r)-u^*_2(z)$ lying within $W^u(u^*_2,-a)$.  Note also that $\|p(\cdot,T_0)\|_{H^1}\leq C r$ by (\ref{eq:CMestimate}).

 The  linear part of \eqref{eq:pevolve} generates a family of $C^0$ semigroups, which we denote $\Phi(t,t_0)$ and which satisfies  
\[ \|\Phi(t,t_0)\|_{H^1(\mathbb{R})}\leq K e^{-\omega (t-t_0)}, \]
for some $0<\omega<\rho$ and some $K\geq 1$.  This estimate holds  for any $0<r<r_0$ with $K$ and $\omega$ independent of $r$. 

If  $p\in B_\kappa(0)$ (the ball of radius $\kappa$ in $H^1(\mathbb{R})$) there exists a constant $C(\kappa)$ such that $\|N(\cdot,p)\|_{H^1}\leq C(\kappa) \|p\|_{H^1}^2$.  Consider the integral form of (\ref{eq:pevolve}):
\begin{equation}\label{eq:pimplicit}
p(z,t)=\Phi(t,T_0)p(z,T_0)+r\int_{T_0}^t  \Phi(t,s)\left[g(\gamma(s)) \frac{\partial}{\partial z} u^*_2(z)\right]\mathrm{d}s +\int_{T_0}^t \Phi(t,s) N(p(z,s),z)\mathrm{d}s .\end{equation}

Since the profile $u^*_2(z)$ is obtained as the solution to a non-autonomous (in $z$) ODE and $f(u,H(z))$ is smooth in both arguments it follows that the profile $u^*_2$ inherits the same degree of regularity and $\partial_zu^*_2$ is therefore an element of $H^1(\mathbb{R})$.  Let 
\[ \Theta(t)=\sup_{T_0\leq \tau \leq t} \|p(\cdot,\tau)\|_{H^1(\mathbb{R})}. \]
From (\ref{eq:pimplicit}) we obtain
 \begin{align} \nonumber  \| p(\cdot,\tau)\|_{H^1} \leq K e^{-\omega(\tau-T_0)}  
\|p(\cdot,T_0)\|_{H^1}&+r\int_{T_0}^\tau K e^{-\omega (\tau-s)} |g(\gamma(s))|\|\partial_z u^*_2\|_{H^1} \mathrm{d}s \\
&+C(\kappa) \int_{T_0}^\tau K e^{-\omega (\tau-s)} \|p(\cdot,s)\|^2_{H^1}\mathrm{d}s, \end{align}
so that there exists a constant $C_2$, independent of $t$, such that 
\begin{equation} \label{eq:thetabd}
    \Theta(t) \leq \Gamma(t)+C_2\Theta(t)^2, 
\end{equation}
where 
\[ \Gamma(t) =K  
\|p(\cdot,T_0)\|_{H^1} +r\sup_{T_0\leq \tau\leq t}\left\{ \int_{T_0}^\tau K e^{-\omega (\tau-s)} |g(\gamma(s))|\|\partial_z u^*_2\|_{H^1} \mathrm{d}s\right\}.\]
Observe that $\Gamma(T_0)=K \|p(\cdot,T_0)\|_{H^1}$ and  $0<\Gamma(t)<Cr$ for some constant $C$ and any $t\geq T_0$. Additionally, note that  $\Theta(T_0)=\|p(\cdot,T_0)\|_{H^1}$.  These facts combine to imply that  
\[ C_2\Theta(T_0)^2-\Theta(T_0)+\Gamma(T_0)>0.\]
Now fix $r_0$ sufficiently small so that the polynomial $C_2\Theta^2-\Theta+\Gamma(t)$ (see (\ref{eq:thetabd})) has the real root
\[ \Theta_-(t)=\frac{1}{2C_2}\left(1-\sqrt{1-4C_2\Gamma(t)}\right), \]
for all $t\geq T_0$ and furthermore that $\Theta_-(t)<\kappa$. The PDE is well posed and therefore $\Theta(t)$ is continuous.  The inequality (\ref{eq:thetabd}) implies that
\[ \Theta(t)\leq \Theta_-(t)\]
for all $t>T_0$.  Since $\Gamma(t)\leq C_1r$ this bound implies \eqref{eq:convergencersmall} for some $C>0$.  Finally, for $t$ sufficiently large we have that the orbit enters a neighborhood of the steady state $(u^*_2,a)$.  Asymptotic stability of this steady state then establishes convergence in the limit as $t\to\infty$ in the $H^1$ norm.   This implies convergence in $L^\infty$, and the Lemma follows.

\end{proof}

\subsection{Small displacement}

As discussed above, an interesting result in our case is that if the displacement is too small then no tipping occurs regardless of the rate. This is an effect that only occurs in the situation we consider in which the habitat moves from one position to another, and cannot be seen in the case of an unbounded linear drift;
as seen in \cite{berestycki2009can} and \cite{hasan2023rate}.

\begin{theorem} {\bf [Minimum displacement requirement]} \label{lem:mindisp} If the displacement $d=2a\ll 1$ then the pullback attractor $\mathcal{U}=U(t)$ converges to $u_2^*$ as $t\to\infty$ for any $r>0$ and no R-tipping occurs.
\end{theorem}

\begin{proof} The proof mimics that of Proposition~\ref{prop:tippingno}.  Starting with the equation (\ref{eq:main}), we let $z=x-a$ (note this  is a different $z$ from that used before as now the time-dependent shift will be dominated) to obtain the PDE
\begin{equation} u_t=u_{zz}+f\left(u,H(z-\gamma(t)+a)\right). \label{eq:uforsmalldisp}\end{equation}
Let $u(z,t)=u^*_2(z)+p(z,t)$.  Then the perturbation $p(z,t)$ obeys the equation
\begin{equation} p_t=\mathcal{L}p +R(z,p)+N(z,p),\label{eq:pevolforfastshift} \end{equation}
where 
\begin{eqnarray*}
    \mathcal{L}&=&\partial^2_z+f_u(u^*_2(z),H(z)), \\
    R(z,p,t)&=&f(u^*_2+p,H\left(z-\gamma(t)+a)\right)-f(u^*_2+p,H(z)), \\
    N(z,p)&=& f(u^*_2+p,H(z))-f(u^*_2,H(z))-f_u(u^*_2,H(z))p.
\end{eqnarray*}
If $p\in B_\kappa(0)\subset H^1(\mathbb{R})$ then, as in the proof of Proposition~\ref{prop:tippingno}, it holds that there exist a $C(\kappa)$ such that  
$\|N(\cdot,p)\|_{H^1(\mathbb{R}}\leq C(\kappa)\|p\|^2_{H^1(\mathbb{R})}$.  Define
\[ \Delta H(z,t)=H\left(z-\gamma(t)+a)\right)-H(z).\] 
Since $H$ is smooth there exists a constant $C>0$, independent of $r$,  such that 
\begin{equation} \|\Delta H (\cdot,t)\|_{L^\infty}<Ca, \label{eq:DeltaHbound} \end{equation}
while convergence of $\gamma(t)\to a$ as $t\to\infty$ implies that 
\[ \|\Delta H(\cdot,t)\|_{L^\infty}\to 0\],
as $t\to\infty$. Suppose that $p\in B_\kappa(0)\subset H^1(\mathbb{R})$.  Since the function  $f$ is smooth in each of its arguments and therefore locally Lipschitz there exists a constant $L(\kappa)$ for which 
\[ \|R(z,p,t)\|_{H^1}\leq L(\kappa) \|\Delta H(z,t)\|_{L^\infty}\leq M(\kappa)a, \]
where the second inequality follows from (\ref{eq:DeltaHbound}).  
 We then write the PDE (\ref{eq:pevolforfastshift}) in the integral form
\[ p(z,t)=e^{\mathcal{L}(t-T_0)}p(z,T_0)+\int_{T_0}^t e^{\mathcal{L}(t-s)}R(z,p,s)\mathrm{d}s +\int_{T_0}^t e^{\mathcal{L}(t-s)}N(z,p)\mathrm{d}s.\]
so that for any $T_0\leq \tau\leq t$,
\begin{eqnarray}
    \|p(\cdot,\tau)\|_{H^1}\leq Ke^{-\rho (\tau-T_0)}\|p(\cdot,T_0)\|_{H^1} &+&KM(\kappa) a \int_{T_0}^\tau e^{-\rho(\tau-s)}\mathrm{d}s  \nonumber \\
    &+&KC(\kappa) \int_{T_0}^\tau e^{-\rho(\tau-s)}\|p(\cdot,s)\|^2_{H^1}\mathrm{d}s.   
\end{eqnarray}
Define $\Theta(t)$ as in the proof of Proposition~\ref{prop:tippingno}.  Then $\Theta(t)$ must satisfy the inequality
\begin{equation} \Theta(t)\leq K\Theta(T_0)+K_0a+K_2\Theta(t)^2. \label{eq:inequalityTheta} \end{equation}
Since the PDE (\ref{eq:uforsmalldisp}) is well posed $\Theta(t)$ is continuous in $t$. Consequently (\ref{eq:inequalityTheta}) must hold for all $t>T_0$ for which $\Theta(t)<\kappa$.  If $\|p(\cdot,T_0)\|_{H^1}=\Theta(T_0)\leq \frac{1}{4K}-\frac{K_0}{K}a$ then actually  $\Theta(t)\leq \Theta_-$ for all $t$ where
\[ \Theta_-=\frac{1}{2K_2}\left(1-\sqrt{1-4K_0a-4K\Theta(T_0)}\right).\]

To complete the proof, for any $r>0$, we can take initial data on the local unstable manifold $W^u(u^*_2,a)$ by taking $T_0$ sufficiently negative so that $-a<\gamma(T_0)<-a+\delta$; see Lemma~\ref{prop:pullback}.  By taking $T_0$ smaller we can ensure that $\Theta(T_0)<\Theta_-(T_0)<\kappa$.  The fact that $\|R(z,p,t)\|_{H^1}\to 0$ as $t\to\infty$ combined with the local asymptotic stability of the state $(u^*_2,a)$ implies convergence to $u^*_2$ as $t\to\infty$.

\end{proof}

\subsection{Large displacement: proof of Theorem \ref{prop:main}}

To prove Theorem \ref{prop:main}, we need end-point tracking of the moving base state
for $r$ small and tipping to extinction for $r$ large enough. As shown in Proposition \ref{prop:tippingno}, end-point tracking holds for $r\ll 1$ independently of the displacement $d$. We need to show, therefore, that tipping to extinction occurs for fast shifts  when the displacement is large enough

As a precursor to the analysis of the fast shift $(r\gg 1)$ regime, we consider the case of an instantaneous habitat shift where the favorable habitat is instantaneously shifted from being centered at $x=-a$ to $x=a$.  It is more convenient to analyze this system in the original $x$ coordinate frame where we consider the following IVP:
\begin{equation}\label{eq:instantIVPina}
\begin{split}
    u_t&= u_{xx} +f(u,H(x-a)) \\
     u(x,0)& =u_2^*(x+a).
\end{split}
\end{equation}
In fact, since we plan to use the habitat displacement as a parameter we will actually consider the equation
\begin{equation}\label{eq:instantIVPintilde}
\begin{split}
    u_t&= u_{xx} +f(u,H(x)) \\
     u(x,0)& =u_2^*(x+2a),
\end{split}
\end{equation}
where we have effectively translated the spatial domain in (\ref{eq:instantIVPina}) by $a$ units so that the stable pulse is centered at $x=0$ as $t\to\infty$.  To avoid overburdening notation, we continue to use $x$ as the spatial coordinate.  

\begin{lemma}{\bf [Extinction for instantaneous habitat shift]} \label{lemma:extinctionrinfinity}    Let $d=2a>0$ be sufficiently large.  Then the solution of the initial value problem (\ref{eq:instantIVPintilde}) satisfies
\[ \lim_{t\to \infty} \sup_{x\in\mathbb{R}} \left| u(x,t)\right| = u_0^* = 0.\]
\end{lemma}
\begin{proof}
     The proof follows from the comparison principle for scalar reaction-diffusion equation after construction of a suitable super-solution.    Recall the definition of the per-capita growth rate: $\tilde{f}(u,H(x))=\frac{f(u,H(x))}{u}$.  
     We construct the super-solution using the following facts. First, for $x$ sufficiently large the per-capita growth rate of the population is strictly negative for all positive population densities.  Second, for any $x$ the per capita growth rate is strictly negative for all sufficiently small population densities.

    To the first point, since $H(x)\to 0$ as $x\to\pm\infty$ there exists a value $x^*<0$ for which the per-capita growth rate is strictly negative for all positive $u$.  This implies that for $x<x^*$ we have that $f(u,H(x))<-\kappa u$ for some $\kappa>0$.  To the second point, we have that there exists a  $\zeta>0$ such that $f(u,H(x))<0$ for all $x\in \mathbb{R}$.    
    We claim the existence of a monotone decreasing function function $\psi_{K,\delta}(x):\mathbb{R}^-\to \mathbb{R}$ such that 
    \[ \psi_{K,\delta}(x)\to K \text{ as } x\to -\infty, \]
    and  $\psi_{K,\delta}(0)=\psi_{K,\delta}'(0)=\psi_{K,\delta}''(0)=0$ while $|\psi_{K,\delta}''(x)|\leq \delta$ for all $x<0$. 
    Define 
    \[ \overline{u}_{K,\delta,\epsilon}{(x,t)}=\left\{ \begin{array}{cc} \zeta & x\geq x^* \\
    \zeta+\psi_{K,\delta}(x-x^*)e^{-\epsilon t} & x<x^* \end{array}\right. , \]
    with $\zeta$ defined as above.

    We will require $\epsilon>0$ and that $K$ is fixed to satisfy  $K+\zeta >u^*_2(0)=\max_{z\in\mathbb{R}} u^*_2(z)$. We therefore proceed with $K$ and $x^*$ fixed and will show that there exists a $\delta$ and $\epsilon$ sufficiently small so that $\overline{u}_{K,\delta,\epsilon}$ is a super-solution.  

    Let 
    \[ \mathcal{N}(u{(x,t)})=u_t-u_{xx}-f(u,H(x)).\]
    A function  $u$ is a classical super-solution if $u\in C^2(\mathbb{R},\mathbb{R}^+)$ and  $\mathcal{N}(u)\geq 0$ for all $(x,t)\in \mathbb{R}\times \mathbb{R}^+$. Since $\overline{u}_{K,\delta,\epsilon}\in C^2(\mathbb{R},\mathbb{R}^+)$ we focus on the sign of $\mathcal{N}(\overline{u}_{K,\delta,\epsilon})$.

    For $x>x^*$ the candidate super-solution is constant and we obtain  directly that
    \[ \mathcal{N}(\overline{u}_{K,\delta,\epsilon}) = -f(\zeta,x) \geq 0. \]
    Now consider $x<x^*$ where $f(u,H(x))<-\kappa u$ and compute 
    \[ \mathcal{N}\left( \zeta+\psi_{K,\delta}(x)e^{-\epsilon t} \right)>e^{-\epsilon t} \left[ \left(\kappa-\epsilon\right)\psi_{K,\delta}(x) -\psi_{K,\delta}''(x)\right]+\kappa \zeta. \]
    Since $\kappa$ is fixed, we select $\delta$ and $\epsilon$ sufficiently small so that the right hand side of the above expression is positive for all $x<x^*$.  

We now use this super-solution to prove the lemma.  Since $\lim_{x\to-\infty} \overline{u}_{K,\delta,\epsilon}(x)=K+\zeta>u^*_2(0)$ it follows that there exist a value of $d=2a$ sufficiently large such that $0<u^*_2(x+2a)<\overline{u}_{K,\delta,\epsilon}(x,0)$ for all $x\in\mathbb{R}$.  The comparison principle then implies that $u(x,t)<\overline{u}_{K,\delta,\epsilon}{(x,t)}$ for all $t>0$ and all $x\in\mathbb{R}$.
    
    By continuity of $f$ there exists a $\zeta_1>\zeta$ such that  $f(\zeta_1,H(x))<0$ for all $x\in\mathbb{R}$ and all $0\leq u \leq \zeta_1$.    Uniform convergence in time of $\overline{u}_{K,\delta,\epsilon}{(x,t)}$ to $\zeta$ implies that there exists a time $t_1$ such that $u(t_1,x)<\zeta_1$ for all $x$.    We then employ a secondary super-solution $\overline{u}_1(x,t)=\zeta_1e^{-\kappa_1 (t-t_1)}$.  Comparison with this super-solution implies uniform exponential convergence of the solution of (\ref{eq:instantIVPintilde}) to the zero extinction state.

\end{proof}

\begin{lemma}{\bf [Extinction for habitat shift with $r\gg 1$]}  Let $d$ be sufficiently large.  Then there exists an $r^+(d)>0$ such that for any $r>r^+$, the pullback attractor $\mathcal{U}(t)$ converges to the extinction state $u_0^*=0$ as $t\to\infty$.
\end{lemma}

\begin{proof}

Consider first differential equation for the ramp function $\gamma'=rg(\gamma)$.  Up to a shift in time we can assume that $\gamma(0)=0$.  Let $\alpha>0$.  Define
\[ T_-=-\frac{1}{r}\int_{-a+\alpha}^0 \frac{1}{g(\gamma)}\mathrm{d}\gamma, \quad  T_+=\frac{1}{r}\int_0^{a-\alpha} \frac{1}{g(\gamma)}\mathrm{d}\gamma. \]
Then $\gamma(T_-)=-a+\alpha$ and $\gamma(T_+)=a-\alpha$.
 When $r$ is large the ramp function transitions rapidly from $-a$ to $a$. The transition time from $-a+\alpha$ to $a-\alpha$ is  $\mathcal{O}\left(r^{-1}\right)$ as $r\to\infty$. 
Our proof now involves three steps:
\begin{enumerate}
    \item Restricted to the time interval  $(-\infty,T_-)$ the ramp function $\gamma(t)$ deviates only a small distance from $-a$.  We use essentially the same argument as employed in the proof of Theorem~\ref{lem:mindisp} to obtain $L^\infty$ control on this evolved profile. 
    \item During the time interval $(T_-,T_+)$ the habitat ramp function quickly evolves shifting the favorable habitat from near $x=-a$ to near $x=+a$.  However, this shift occurs in an $\mathcal{O}\left(\frac{1}{r}\right)$ time interval.  Thus, for $r$ sufficiently large the solution profile can only change a small amount during this time period. 
    \item On the interval $(T_+,\infty)$ the ramp function is close to $+a$.  We therefore use our pointwise control on the forward evolution of the pullback attractor at $t=T_+$ combined with the super-solution constructed in Lemma~\ref{lemma:extinctionrinfinity} to show pointwise uniform convergence of the solution to zero.      
\end{enumerate}

\noindent{\bf Step 1} Consider any fixed $r>0$. Let $\gamma_0$ satisfy
$-a<\gamma_0<-a+\delta$ and define $T_0$ so that $\gamma(T_0)=\gamma_0$.  If $T_0<T_-$ then we need to control the evolution of the solution up to the time $t=T_-$. To do so, we essentially replicate the proof of Theorem~\ref{lem:mindisp}.  Reverting from the coordinate $z=x-\gamma(t)$ to $x$ we write $u_1(x,t)=u^*_2(x+a)+p_1(x,t)$ where the perturbation $p_1(x,t)$ is a solution of
\begin{equation}
\frac{\partial p_1}{\partial t} = p_{1,xx}+f_u(u^*_2(x+a),H(x+a))p_1 +R_1(p_1,x,t)+N_1(p_1,x), \nonumber
\end{equation}
where
\begin{eqnarray}
    R_1(p_1,x,t)&=& f\left(u^*_2(x+a)+p_1,H(x-\gamma(t))\right) - f\left(u^*_2(x+a)+p_1,H(x+a)\right)  \nonumber \\
    N_1(p_1,x)&=& f\left(u^*_2(x+a)+p_1,H(x+a)\right)-f_u\left(u^*_2(x+a),H(x+a)\right)p_1  \nonumber\\ 
    &-&f\left(u^*_2(x+a),H(x+a)\right), 
\end{eqnarray}
with initial data 
\[ p_1(x,T_0)=\phi(x-\gamma_0,\gamma_0,r)-u^*_2(x+a),\]
recall Lemma~\ref{prop:pullback}.  Using that $0<\gamma(t)+a<\alpha$ and repeating the analysis in Theorem~\ref{lem:mindisp} we obtain -- for $\alpha$ sufficiently small -- that 
\[ \|p_1(\cdot,T_-)\|_{H^1}\leq C_\alpha \alpha +C_1\|p_1(\cdot,T_0)\|_{H^1}, \]
for some constants $C_\alpha$ and $C_1$, independent of $r$.

{\bf Step 2} Let us now consider the interval $t\in [T_-,T_+]$.  Let us denote the solution in this interval as $u_2(x,t)$ and expand $u_2(x,t)=u^*_2(x+a)+p_2(x,t)$ with initial data $p_2(x,T_-)=p_1(x,T_-)$.  Note that $p_2$ then satisfies the PDE
\[ \frac{\partial p_2}{\partial t} = p_{2,xx}+f\left(u^*_2(x+a)+p_2,H(x-\gamma(t))\right)-f\left(u^*_2(x+a),H(x+a)\right).   \]
Writing the integral form of the solution as 
\[ p_2(x,t)= e^{\partial^2_x t}p_2(x,T_-)+\int_{T_-}^t e^{\partial^2_x (t-\tau)}R_2(x,\tau)\mathrm{d}\tau + \int_{T_-}^t e^{\partial^2_x (t-\tau)}N_2(p_2,x,\tau)\mathrm{d}\tau,
\]
where 
\begin{eqnarray*}
    R_2(x,t)&=& f\left(u^*_2(x+a),H(x-\gamma(t))\right)-f\left(u^*_2(x+a),H(x+a)\right) \\
    N_2(p_2,x,t) &=& f\left(u^*_2(x+a)+p_2,H(x-\gamma(t))\right)-f\left(u^*_2(x+a),H(x-\gamma(t))\right).
\end{eqnarray*}
For $p\in B_\kappa(0)$, we use the bounds $\|N_2(p_2,x,t)\|_{H^1}\leq M \|p_2\|_{H^1}$ and $\|R_2(x,t)\|_{H^1}\leq L$ to obtain that
\[ \|p_2(\cdot,t)\|_{H^1} \leq \|p_2(\cdot,T_-)\|_{H^1} +(t-T_-) L +M \int_{T_-}^t \|p_2(\cdot,\tau)\|_{H^1}\mathrm{d}\tau. \]
An application of Gr\"onwall's inequality then yields 
\[ \|p_2(\cdot,T_+)\|_{H^1} \leq \left(\|p_2(\cdot,T_-)\|_{H^1} +L (T_+-T_-)\right)e^{M (T_+-T_-)}. \]
In particular, for any $C>1$, there exists a $r_+(C,\alpha)$ sufficiently large such that \[ \|p_2(\cdot,T_+)\|_{H^1} \leq C\|p_2(\cdot,T_-)\|_{H^1}, \]
for all $r>r_+(C,\alpha)$.  

{\bf Step 3} For the sake of exposition take $C=2$ in the preceding equation.  We have therefore established that at time $t=T_+$ the forward evolution of $\mathcal{U}(t)$ has the form $u^*_2(x+a)+p_2(x,T_+)$ where
\[ \|p_2(x,T_+)\|_{H^1} \leq 2C_\alpha \alpha +2C_1\|p_1(\cdot,T_0)\|_{H^1}. \]
This estimate is valid for any $r>r_+(2,\alpha)$  with $C_\alpha$ and $C_1$ independent of $r$.  Since $L^\infty(\mathbb{R})$ is embedded continuously in $H^1(\mathbb{R})$ we obtain pointwise control on the forward evolution of $\mathcal{U}(t)$ through some combination of taking $T_0<0$ sufficiently large or by taking $\alpha$ sufficiently small.    

On the time interval $[T_+,\infty)$ we have that $a-\alpha<\gamma(t)<a$.  We now modify the argument from Lemma~\ref{lemma:extinctionrinfinity} to complete the proof.  For $\alpha$ fixed, there exists $x_*$ and a $\kappa>0$ such that $f(u,H(x-\gamma(t)))<-\kappa u$ for any $x<x_*$ and any $t\geq T_+$.  We also obtain that  there exists a $\zeta>0$ such that $f(\zeta,H(x-\gamma(t))<0$ for all  $x\in\mathbb{R}$ and $t\geq T_+$.  Now define -- as in the proof of Lemma~\ref{lemma:extinctionrinfinity} -- the function
\[  \overline{u}_{K,\delta,\epsilon}{(x,t)}=\left\{ \begin{array}{cc} \zeta & x\geq -x_* \\
    \zeta+\psi_{K,\delta}(x-x_*)e^{-\epsilon (t-T_+)} & x<-x_* \end{array}\right. ,  \]
which is a super-solution for $\epsilon$ and $\delta$ chosen sufficiently small.

{\bf Conclusion} The forward evolution of $\mathcal{U}(t)$ is now a profile of the form
\[ u^*_2(x+a)+p_2(x,T_+).\]
By taking $a$ sufficiently large and $\alpha$ sufficiently small we can obtain sufficient control to guarantee that 
\[ u^*_2(x+a)+p_2(x,T_+)<\overline{u}_{K,\delta,\epsilon}(x,T_+). \]
We conclude as in Lemma~\ref{lemma:extinctionrinfinity}: an application of the comparison principle shows that eventually $u(t,x)<\zeta_1$ for some $\zeta_1$ for which $f(\zeta_1,H(x-\gamma(t)))<0$ for all $x$.  Then pointwise exponential decay to zero is obtained and the result is established.  
\end{proof}

\section{Determination of critical rate: uniqueness and nondegeneracy}\label{sec:comps}
A central issue is to determine the critical rate, above which there is tipping to extinction and below which there is end-point tracking of the moving stable base state. At this point, we do not know if such a critical rate is unique, and it will depend on the specific structure of the nonlinear term. Our goal is to show this for the specific model problem from Section~\ref{sec:spec_mod},  where the nonlinear term $f(u,H(z))$ is given by \eqref{eq:rterm}, and the habitat function $H(z)$ is given by \eqref{eq:habitatf}.  

\subsection{Computation of the pulses \label{comppulse}}

At the beginning and end of the habitat transition, when it can be considered stationary, the dynamics of the system are governed by the spatially inhomogeneous system  \eqref{main_steady}. Each pulse satisfies,
\eq{u_{zz}-\beta^2 u + \lambda H(z)u^2 -  u^3=0,\quad \lim_{z\to \pm \infty} u(z)=: u_{\pm} = 0.}{\notag}
We write this nonautonomous ODE as a first order system,
\eq{u'&=v,\\v'&= -f(u,H(z)),}{\label{eq:pulse_ODE1}}
where $':=\frac{d}{dz}$ and $f(u,H(z))=-\beta^2 u + \lambda H(z)u^2 -  u^3$. A pulse solution of this PDE is a non-trivial homoclinic orbit to the fixed point at the origin for the nonautonomous ODE~\eqref{eq:pulse_ODE1}. This homoclinic orbit is found numerically by solving a boundary value problem. We pose the problem as a system of double the dimension but on the half line $(-\infty,0]$. The extra dimensions correspond to the ODE with the coordinate change $z\to -z$. The new set of variables are $(u,v,\tilde u,\tilde v)$ and the system is given by
\eq{u'=v,\quad v'= -f(u,H(z)),\quad\tilde u'= -\tilde v',\quad\tilde v'=f(\tilde u,H(-z)).}{\notag}
  To numerically approximate the solution to this ODE system, we truncate the domain to a finite interval $[-Z,0]$; we use $Z = 150$. We choose boundary conditions that correspond to enforcing the solution to decay to the fixed point $(0,0)^T$ in forward and backward time, $z\to \pm \infty$, as described below. For convenience, we introduce a free parameter $\umax$ which shows  up only in the boundary conditions. The boundary conditions consist of (1) the matching conditions $u(0) = \tilde u(0)$, $v(0)=\tilde v(0)$, of (2) the projective boundary conditions $\Pi_S(u(Z),v(Z))^T = 0$ and $\Pi_U(\tilde u(Z),\tilde v(Z))^T= 0$ where $\Pi_S = (1,\frac12\sqrt{4\beta^2})$ is a projection onto the stable subspace of the fixed point $(0,0)^T$ of \eqref{eq:pulse_ODE1} and $\Pi_U=(1,-\frac12\sqrt{4\beta^2})$ is a projection onto the unstable subspace of $(0,0)^T$, and (3) a boundary condition $u(0)-\umax = 0$. This last condition forces the free variable to correspond to the maximum value of the pulse and makes it possible to select the desired pulse solution via an appropriate, rather general initial guess for $\xi$. Without this last boundary condition, the initial guess for the ODE solution has to be very close to the unstable pulse in order for the BVP solver to converge to it. Introducing the free parameter gives us the convenience of not using continuation for our chosen parameter set.

We solve this finite interval boundary value problem using the function BVP5c in MatLab, which uses the four-stage Lobatto IIIa formula. This method adaptively chooses the mesh on which to solve the solution and uses local polynomials to represent the solution, resulting in a $C^1$-continuous approximation with fith order accuracy. Internally, BVP5c treats the free parameter as an additional ODE, $\umax'=0$. For the initial guess to the solution, we set the free parameter $\umax = 0.56$ to find the stable pulse and $\umax = 0.15$ to find the unstable pulse. These guesses work for all values of $L$ that we considered. The initial guess for the solution is built, via taking the derivative and flipping the domain, from $u(z)= \umax\sech(z)$
when seeking the unstable pulse, and when seeking the stable pulse by $u(z)= \frac12 \umax(\tanh(z+L/2)-\tanh(z-L/2))$.  

The nodes on which the solution is solved will be important later. Indeed, when approximating the pullback attractor of \eqref{main_mov}, we discretize \eqref{main_mov} in the spatial direction with nodes that correspond to those used by BVP5c to solve for the stable pulse, in other words, finding the homoclinic orbit of \eqref{eq:pulse_ODE1} with largest maximum. 

To be more explicit: using a uniform mesh to discretize \eqref{main_mov} in the spatial direction results in derivative approximations that vary significantly in discretization error due to the rapid change in the solution near the origin and nearly flat behavior for large $z$. To use Chebyshev nodes, one has to take many nodes to resolve the solution near $z = 0$. Intuitevely, this makes sense because the Chebyshev nodes are densest where the solution changes the least. We find that using the same nodes to discretize \eqref{main_mov} as used by BVP5c to solve for the stable pulse works very well. Intuitively, this makes sense because the nodes selected by BVP5c are adaptively chosen to minimize the pulse solution error, which in turn depends on the change in the solution throughout its domain. Indeed, in the limit $t\to -\infty$, the solution of \eqref{main_mov} we are interested in is the stable pulse, and so the spatial discretization for the stable pulse is highly appropriate for \eqref{main_mov} and, in practice, performs well throughout the whole time interval. 

We plot the pulse solutions in Figure~\ref{fig:pulses}. We see visually (and verify numerically) that these pulses satisfy (H1), that is, $u_2^*(z)> u_1^*(z)$ for all $z$.

Note that Hasan et al.~\cite{hasan2023rate} provide an alternative method for computing pulse solutions for a static habitat as heteroclinic orbits between two fixed points of a suitably compactified autonomous ODE. Their method also extends to a linearly drifting habitat, where the pulse solutions are not specifically steady states, but traveling waves.

\begin{figure}[t]
    \centering
    \includegraphics[width=0.49\textwidth]{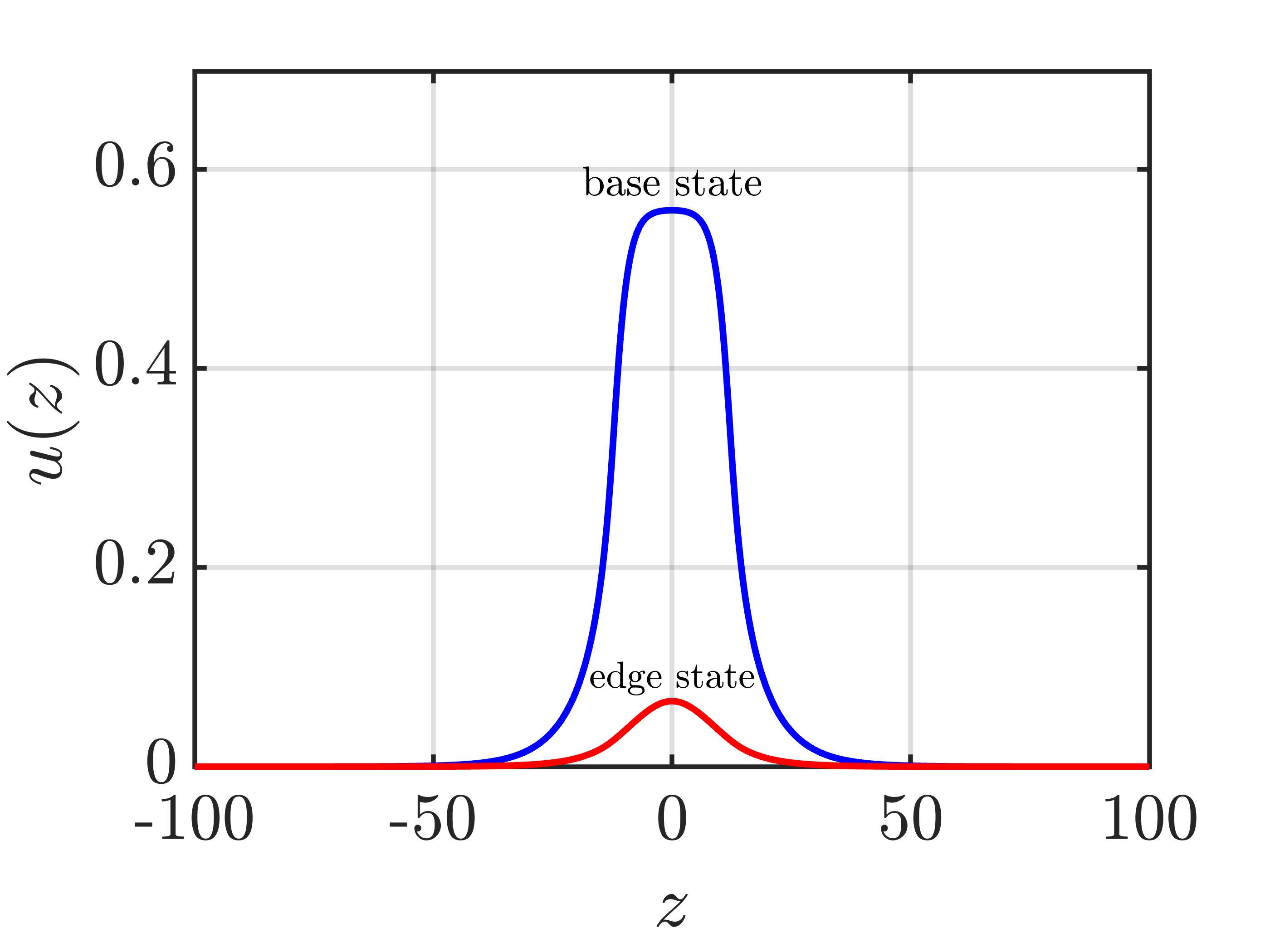}
    \caption{ Comparison of pulse profiles showing the larger-amplitude stable pulse ($u^*_2$, blue solid) and smaller-amplitude unstable pulse ($u^*_1$, red solid).
    Both profiles are computed for $\beta = 0.15$, $\lambda = 4 \beta$ and $L = 25$ using the boundary value problem described in Section~\ref{comppulse}.
    }
    \label{fig:pulses}
\end{figure}

\subsection{Spectral stability of the steady pulses}

Linearizing about a pulse solution $u^*(z)$ obtained as described in \ref{comppulse},  we obtain the eigenvalue problem
\begin{equation}\label{evalue}
\lambda p = p_{zz} - \beta^2 p + 8\beta H(z)u^*(z)p - 3{u^*}^2(z)p,
\end{equation}
where  $u^*$ is either the stable pulse $u^*_2$ or the unstable pulse $u^*_1$.

To find the eigenvalues, we follow a standard Sturm-Liouville approach, and use the angular equation for \eqref{evalue}, when it is written as a system. Set $\theta =\arctan\left(\frac{p_z}{p}\right)$, then $\theta (z)$ satisfies 
\begin{equation}\label{eq:theta_ODE}
\theta' = \lambda \cos^2{\theta} - f_u(u^*(z), H(z))\cos^2{\theta} - \sin^2{\theta}
= (\lambda - f_u(u^*(z), H(z)))\cos^2{\theta} - \sin^2{\theta},
\end{equation}
where $'$ denotes $\tfrac{d}{dz}$.  
Equation \eqref{eq:theta_ODE} determines the phase-plane rotation associated with the linearized eigenmodes. We think of $\theta$ as living in $\mathbb{R}$, which corresponds to taking the unwrapped version of arctan. 

It is natural to use a compactification again, but this time in space, rather than time. To do this set $s = h(z) = \tanh\!\left(\frac{z}{2}\right),$
 and $\tilde{H}(s) := H\bigl(h^{-1}(s)\bigr)$. System \eqref{eq:theta_ODE} then transforms into the autonomous form
\begin{equation}\label{eq:nondim2}
\begin{split}
\theta ' &= \left(\lambda - f_u\left(u^*(h^{-1}(s)\right), \tilde{H}(s))\right)\cos^2{\theta} - \sin^2{\theta}, \\
s' &= 1 - s^2.
\end{split}
\end{equation}

The equilibria occur at $s = \pm 1$, corresponding to $z \to \pm\infty$.  
In these limits, $H(z)\to 0$ and $u^*(z)\to 0$, so $f_u(z,H(z)) \to -\beta^2$, 
and the angular equation reduces to
\[
\theta' = (\lambda -(-\beta^2))\cos^2\theta - \sin^2\theta
= (\lambda+\beta^2)\cos^2\theta - \sin^2\theta.
\]
To find fixed points, set $\theta' = 0$, which gives $(\lambda+\beta^2)\cos^2\theta = \sin^2\theta$, 
so $\tan^2\theta = \lambda+\beta^2$.

An eigenvalue occurs when there is a heteroclinic connection in \eqref{eq:nondim2} from a fixed point in $s=-1$ that is attracting (hence corresponds to a subspace of unstable solutions for the system as $z\to-\infty$) and a repelling fixed point in $s=+1$. Note that the attraction and repulsion are for the systems inside $s=\pm 1$, respectively. They have $1$-dimensional stable and unstable manifolds, respectively,  in the full phase space.

When there is an eigenvalue, i.e., a heteroclinic connection exists, the total accumulated angular rotation of the heteroclinic trajectory is a multiple of $\pi$. That multiple of $\pi$, we denote $\theta_d(\lambda)$, and it  corresponds to the eigenvalue index.

We examine two representative parameter sets:

\begin{enumerate}
\item For $\beta = 0.15$, $L = 25$, $\alpha = \beta/2$, and $\max(u^*_1) \approx 0.0657$ (red curve in Figure~\ref{fig:pulses}), we detect eigenvalues at $\lambda \approx -6.8\times 10^{-3}$ and $\lambda \approx 0.026$ above the essential spectrum threshold $-\beta^2$ (blue dashed curve in Figure~\ref{fig:three_panels}, left panel). These eigenvalues indicate that the corresponding standing pulse is unstable, thereby verifying (H3).

\item For the same parameters, $\max(u^*_2) \approx 0.5588$ (blue curve in Figure~\ref{fig:pulses}), the eigenvalue curve (red in Figure~\ref{fig:three_panels}, left panel) remains along the horizontal axis, entirely to the left of $-\beta^2$, confirming spectral stability and (H4).
\end{enumerate}

These results identify the unstable (edge state) and stable (base state) pulses that underlie rate-induced tipping phenomena in \eqref{eq:main}. 

\begin{figure}[t]
    \centering
    \includegraphics[width=0.49\textwidth]{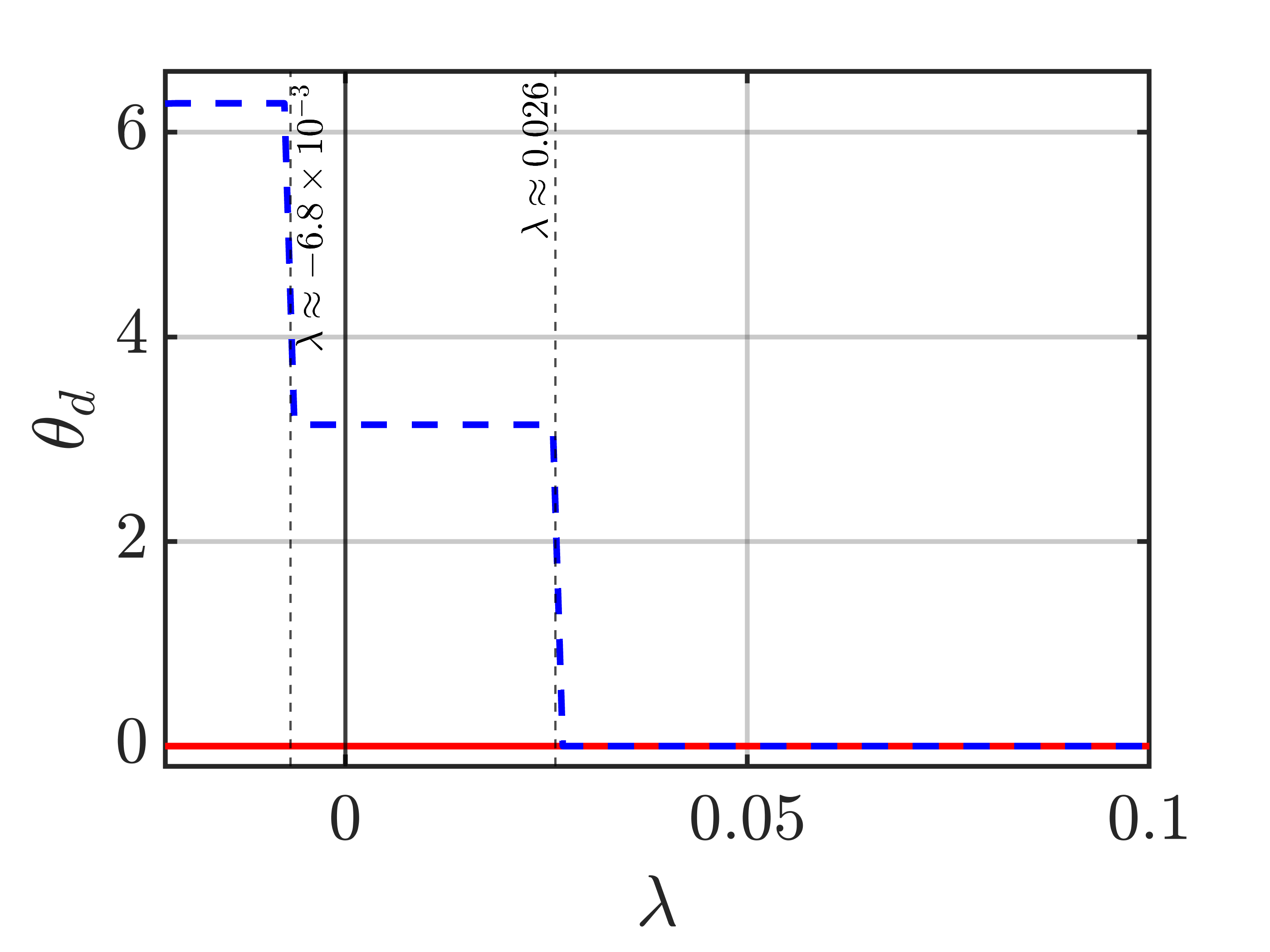}
    \hfill
    \includegraphics[width=0.49\textwidth]{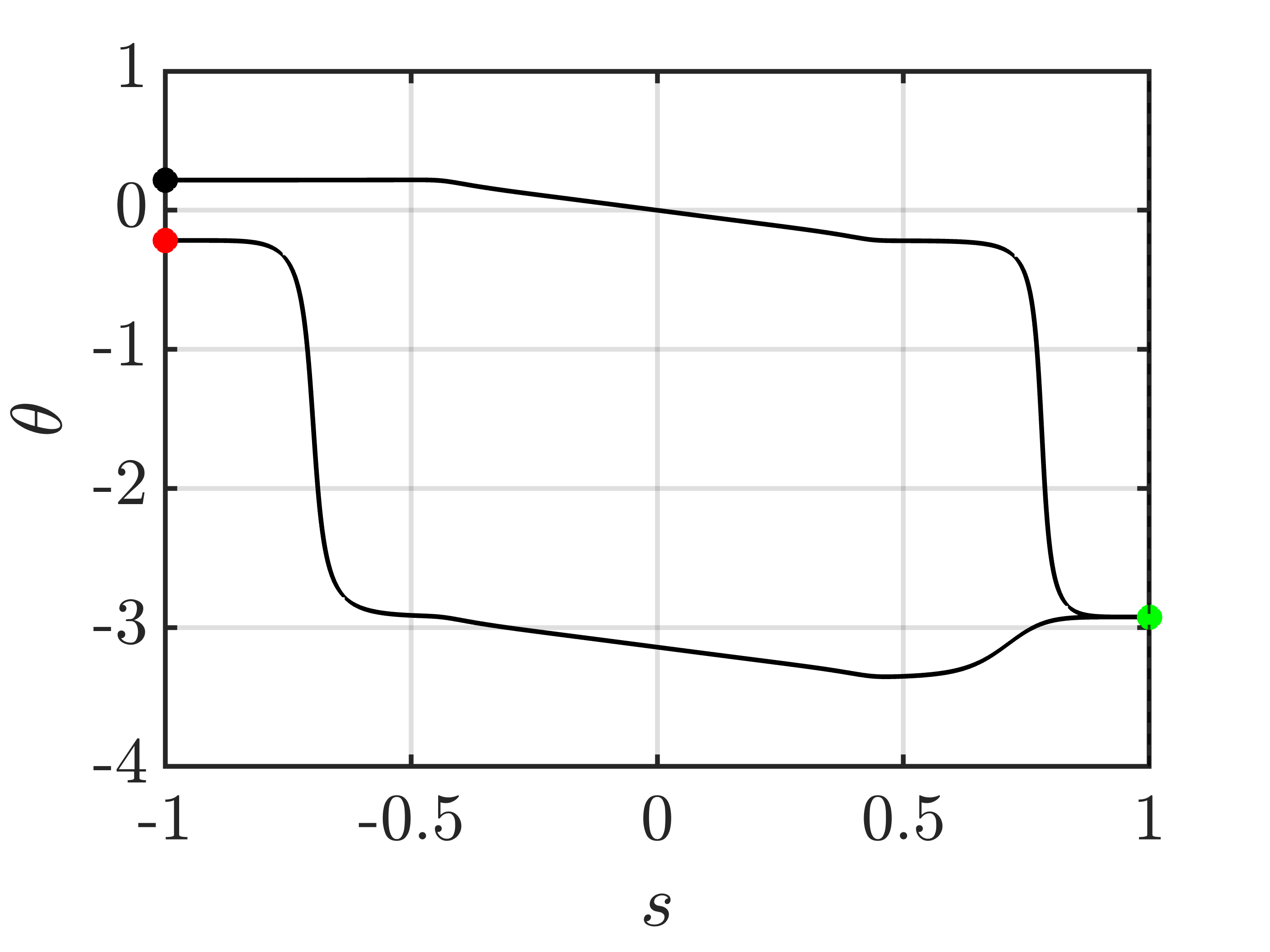}
    \caption{Stability analysis of pulse solutions. 
    Left: Spectral diagnostic via $\theta$-difference, with red solid line representing the stable pulse and blue dashed line representing the unstable pulse; vertical lines mark detected eigenvalues. 
    Right: Phase portrait showing intersecting stable and unstable manifolds connecting fixed points in the $(s, \theta)$ phase plane for $\lambda \approx 0.026$. Both panels use parameters $\beta = 0.15$ and $L = 25$ in \eqref{eq:nondim2}.
    }
    \label{fig:three_panels}
\end{figure}

\begin{figure}[htbp] 
 \begin{center}
$
\begin{array}{lcr}
(a_1) \: \includegraphics[scale=0.4]{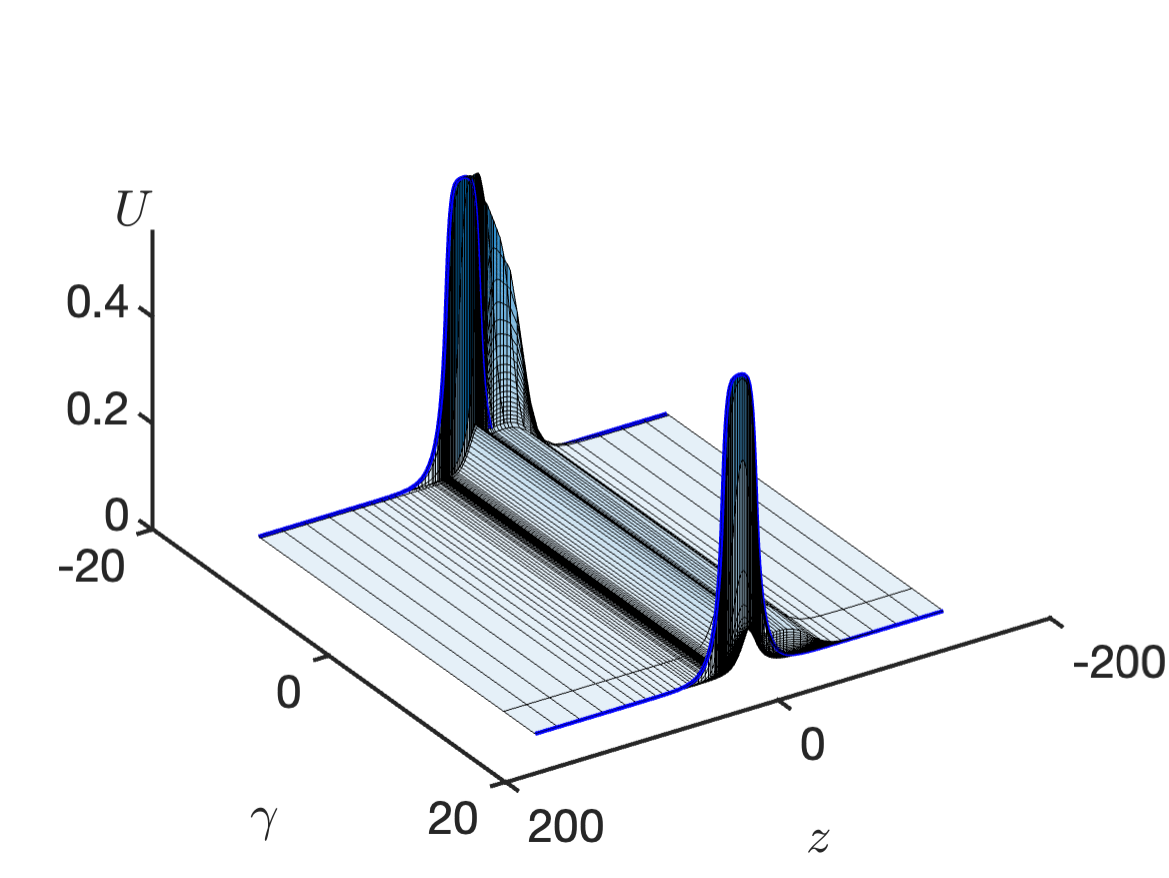}& (b_1) \; \includegraphics[scale=0.4]{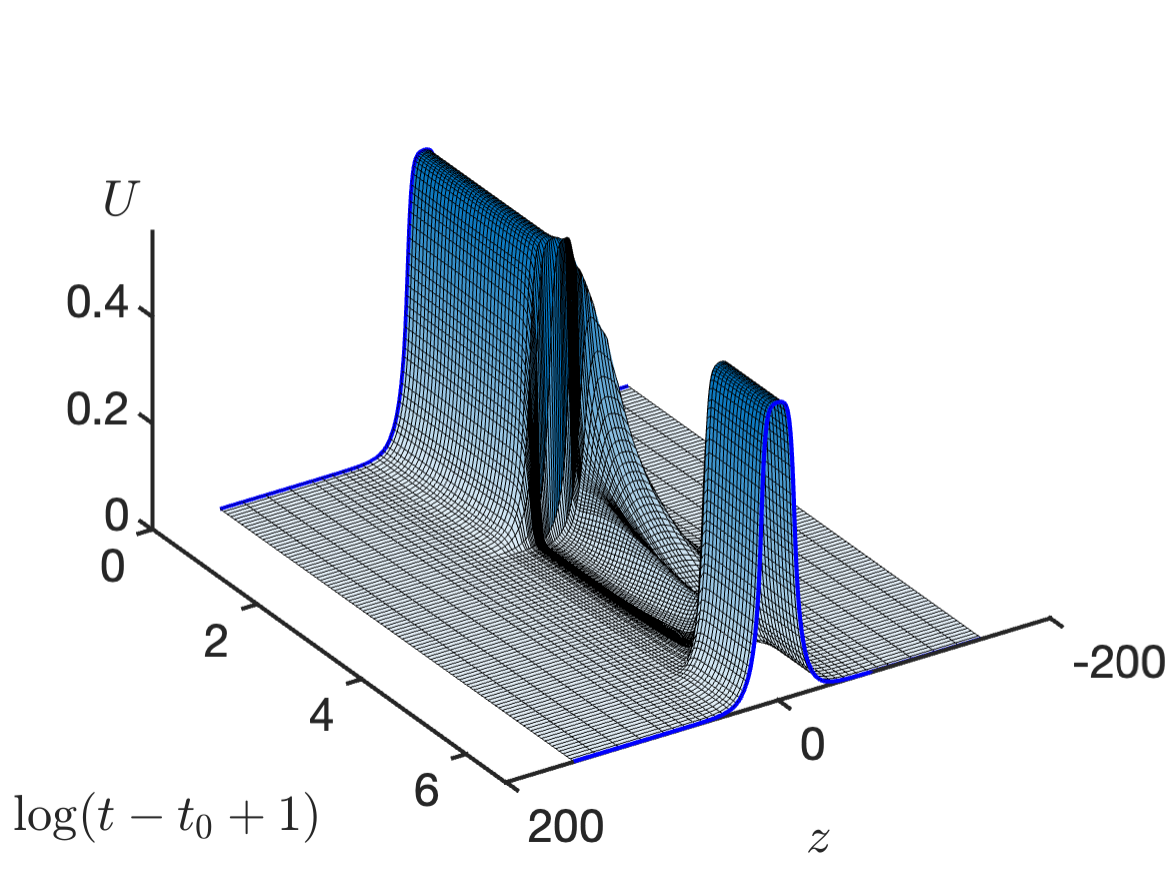} \\(a_2) \: \includegraphics[scale=0.4]{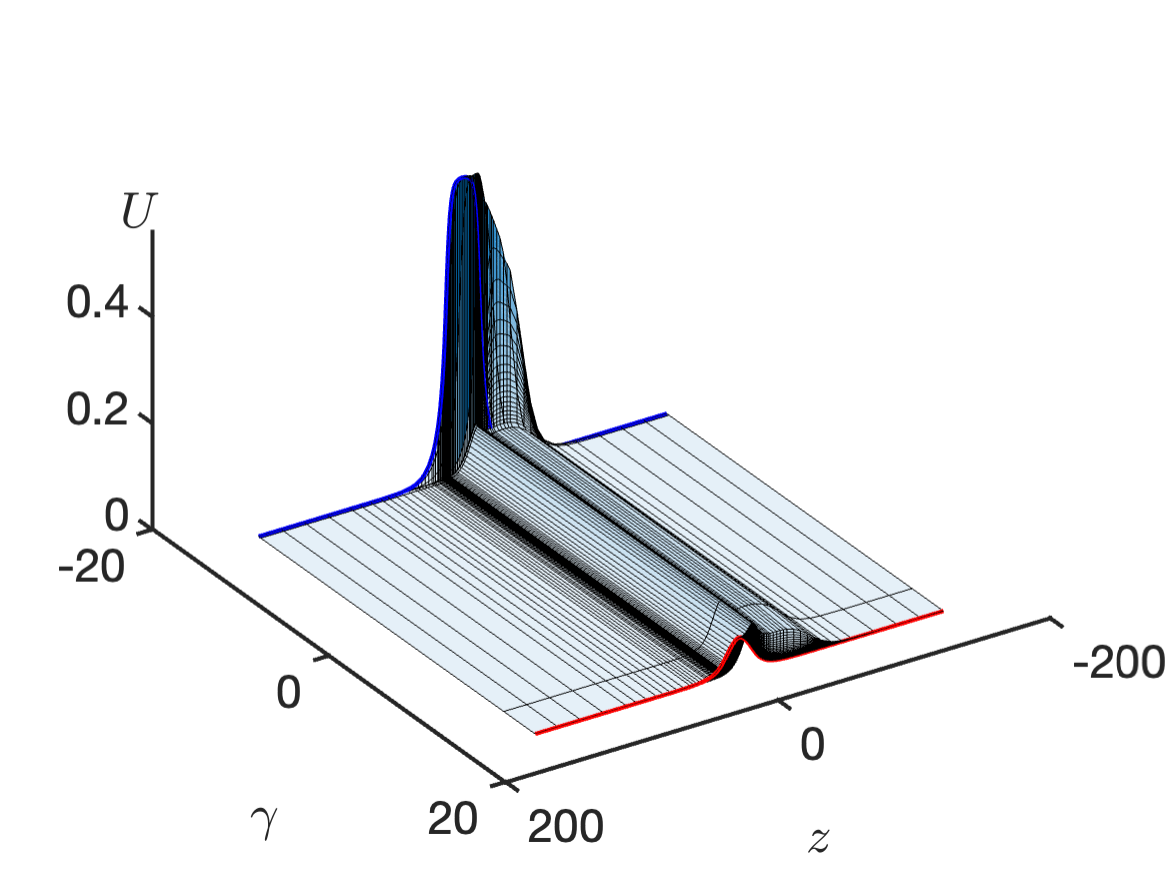} & (b_2) \; \includegraphics[scale=0.4]{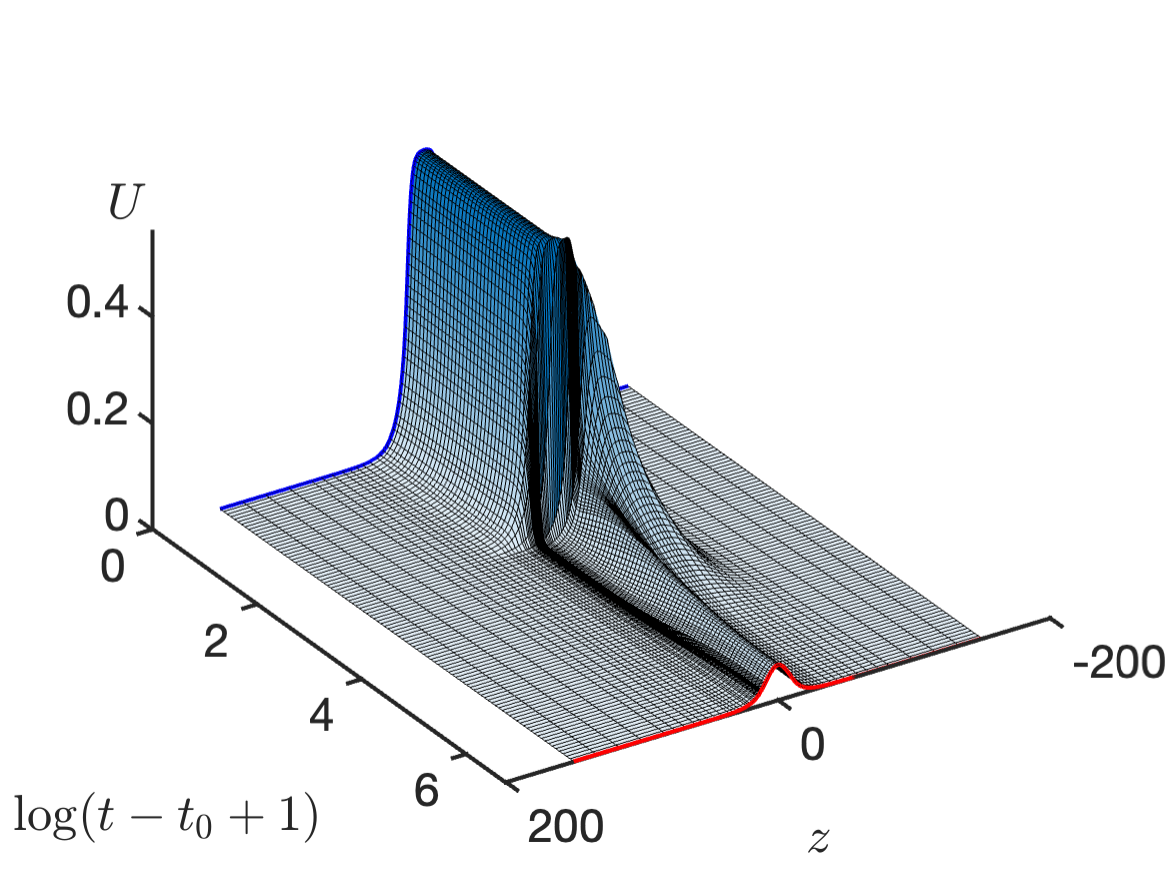} \\
(a_3) \: \includegraphics[scale=0.4]{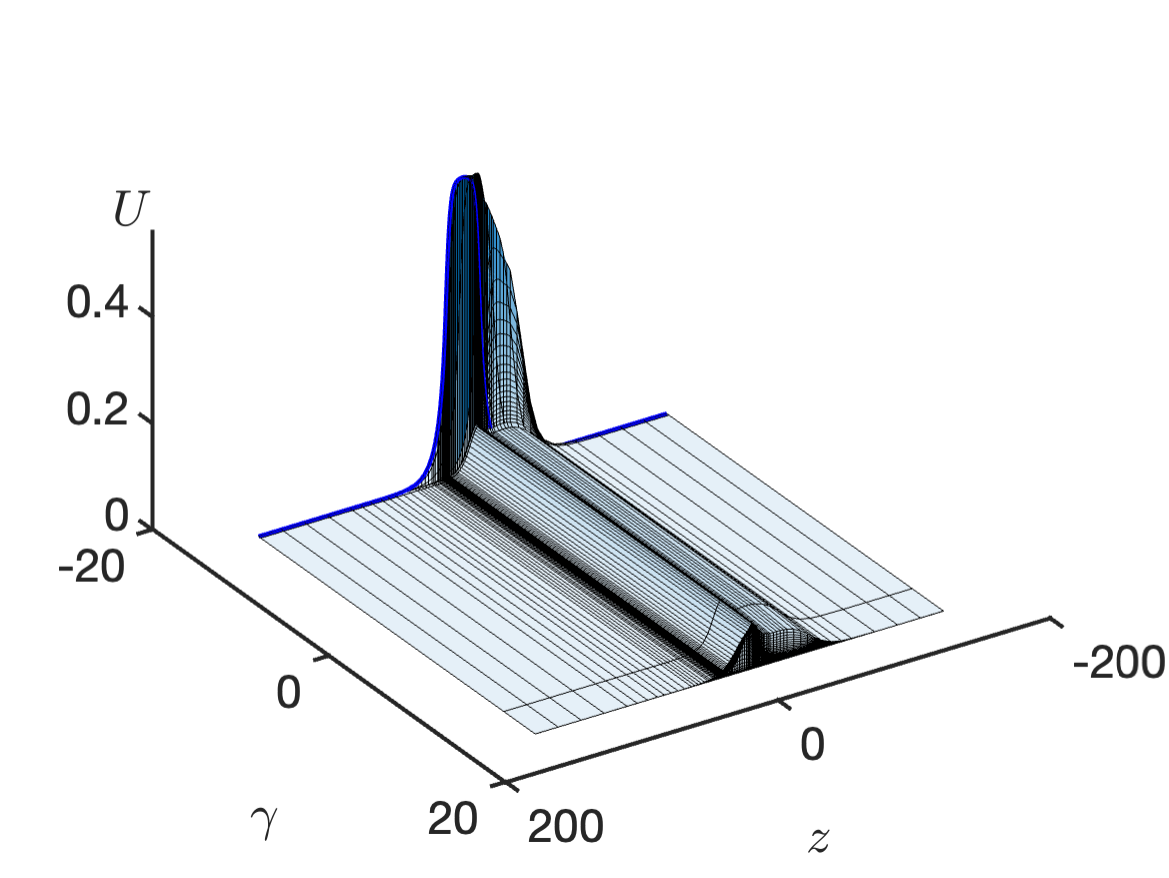}& (b_3) \; \includegraphics[scale=0.4]{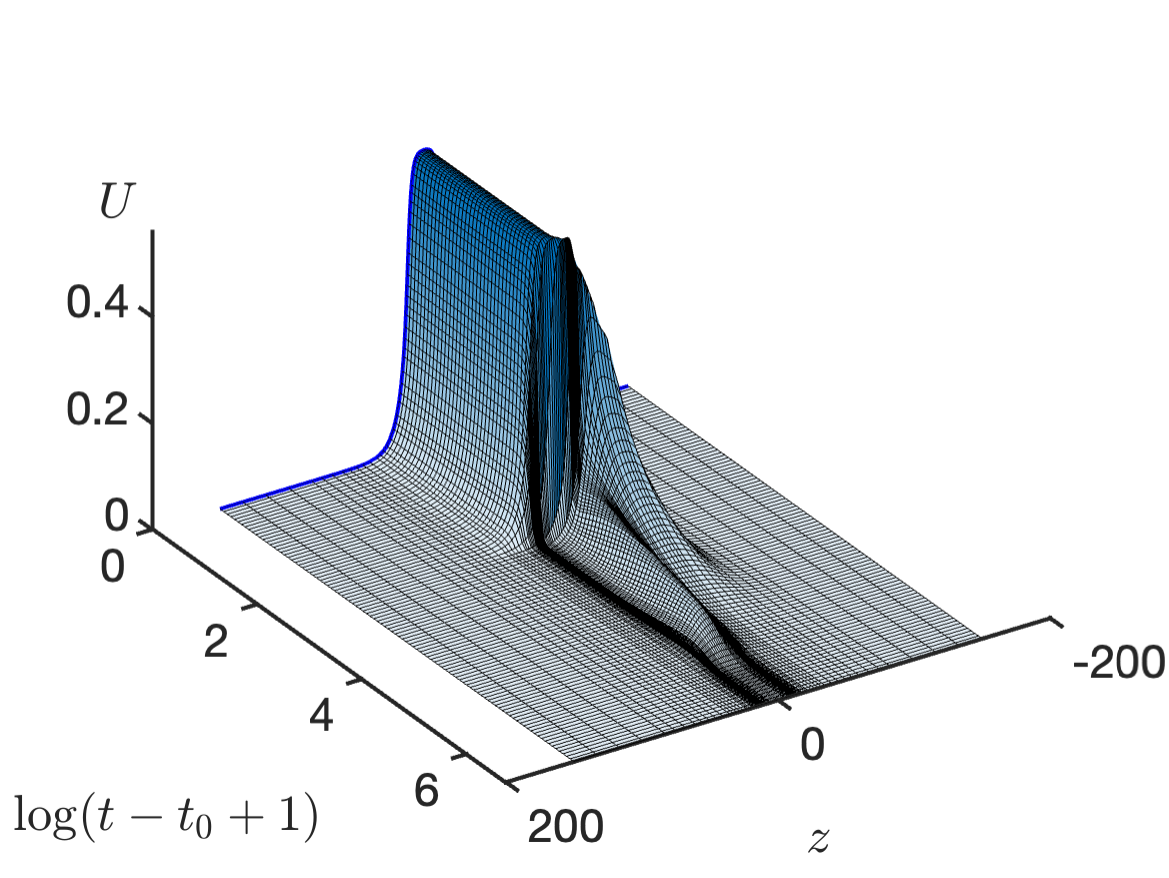} \\
\end{array}
$
\end{center}
\caption{\text
(Left column) Plot, computing the pullback attractor with the shooting approach, of $U$ against $z$ and $\gamma$ and (right column) plot of $U$ against  $z$ and $\log(t-t_0+1)$ for ($a_1$, $b_1$) $r=r_c-0.01$, ($a_2$, $b_2$) $r = r_c$, ($a_3$, $b_3$) $r= r_c+0.01$ where $r_c\approx 0.9670$ is the critical value at which the system undergoes R-tipping. Other parameters are $a = 15.65$, $\beta = 0.15$, $\lambda = 4\beta$, and $L = 25$. 
}
\label{fig:hetero1}
\end{figure}

\subsection{Computation of the pullback attractor\label{subsec:pullback}}

The numerical computation of the pullback attractor $\mathcal{U}$, which depends on the rate $r$, amounts to starting on the unstable manifold $W^u(u_2^*,-a)$, and iterating forward in time. This manifold lives in the phase space of \eqref{main_mov}. To accomplish this, we truncate the time domain of the PDE to a finite interval $[T_L,T_R]$ where $T_L\ll 0\ll T_R$, and we truncate the spatial domain to $[-Z,Z]$, where $Z$ is the truncation value used for solving for the pulses. We then use the method of lines to propagate the PDE solution forward from time $t = T_L$. For this finite time PDE problem, the initial data is given by approximating the unstable manifold $W^u(u_2^*,-a)$ to first order. 

 The method of lines setup is as follows: On the finite time and space domain, we discretize the PDE
\begin{equation}\label{main_mov_copy}
\begin{split}
   u_t=&u_{zz}+rg(\gamma)u_z+f\left(u,H(z)\right),\\
   \gamma_t=&rg(\gamma).
\end{split}
\end{equation}

We choose our spatial grid $\{z_j\}$ to be the set of collocation points used in BVP5c to solve for the stable pulse, as described in Section \ref{comppulse}. 
Let $V_j(t)$ approximate $u(z_j,t)$ where $u$ solves \eqref{main_mov_copy} and $z_j$ is the jth spatial node. Specifically, require that $V = (V_1,...,V_n)$ satisfy the ODE, where we are now using the specific shift function $\gamma(t) =a \tanh rt$ and therefore $g(\gamma) = a-\gamma^2/a$,
\eq{
V_t&= D_{zz} V+r(a-\gamma^2/a)D_zV+f(V_j,H(z_j)),\\
\gamma_t&= r(a-\gamma^2/a).
}{\label{eq:compactified_ODE}}
The differentiation matrices $D_z$ and $D_{zz}$ are the derivative matrices for local interpolation using polynomials of degree 5. Thus, these matrices are band matrices with width 5. For example, if $p_j$ is the polynomial that interpolates the points 
\eq{\{(z_{j-2},V_{j-2}),(z_{j-1},V_{j-1}),(z_j,V_j),(z_{j+1},V_{j+1}), (z_{j+2},V_{j+2})\},}{\notag} then the $j$th row of $D_z$ multiplied by $(V_1,...,V_n)^T$ gives the first derivative of the interpolating polynomial $p_j$ evaluated at $z_j$. The polynomial interpolant uses the five closest grid points for interpolation, which means that for the two grid points closest to a boundary, the interpolation points are not centered. 

The initial value of the ODE system \eqref{eq:compactified_ODE} must correspond to the unstable manifold $W^u(u_2^*,-a)$ of the PDE \eqref{main_mov_copy} posed on the whole line. To accomplish this, we initialize \eqref{eq:compactified_ODE} on the unstable manifold of the fixed point of \eqref{eq:compactified_ODE} that approximates $(u,\gamma) = (u_2^*,-a)$. Specifically, we initialize with a first order approximation of the unstable manifold (ODE) using an appropriate eigenvector of the Jacobian.  To that end, we note that the Jacobian of \eqref{eq:compactified_ODE} is given by
\eq{
J(V,\gamma) &= \mat{D_{zz}+r(a- \gamma^2/a)D_z+Df(V,H(z))
& -(2r\gamma/a)D_z V \\
0& -2r\gamma/a}
}{\label{eq:compactified_jacobian}}
where
\eq{
Df(V,H(z)):= \mat{f_u(V_1,H(z_1))&0&0&\hdots\\
0&f_u(V_2,H(z_2))&0&\hdots\\
\vdots&0&\ddots&0\\
0&\hdots&0&f_u(V_n,H(z_n)) }.
}{\notag}

We evaluate the Jacobian at the fixed point $(V_s,-a)$ of \eqref{eq:compactified_ODE} that corresponds to the stable pulse at location $z = -a$, and then we find the eigenvector $X$ corresponding to the eigenvalue with positive real part. The ODE system is then initialized at $(V_s,-a)^T+\eps X$ with the choice $\eps = 1e-8$. That is, the ODE system is initialized with the first order approximation of the unstable manifold of the fixed point that corresponds to the stable pulse centered at $z = -a$. Thus, the method of lines is initialized with a first order approximation of the unstable manifold (ODE) of the fixed point approximating the stable pulse, which in turn provides an approximation of the unstable manifold of the PDE \eqref{main_mov_copy}.

We solve the ODE system (\ref{eq:compactified_ODE}) on the time interval $[-50,800]$ using the Matlab routine ODE15s. The numerical approximation of the pullback attractor is depicted in Figure \ref{fig:hetero1} for several parameter values depicting scenarios of survival, extinction and the critical rate case.  Note that these are depicted in the moving frame; in the original frame, a shift in $x$ as $\gamma$ increases would be seen.  Further note that we plot time on a log scale in Figures \ref{fig:hetero1} ($b_1$)-($b_3$) in order to see with better detail the transition behavior of the pullback attractor between its time-asymptotic states. Figure \ref{fig:hetero1}($b_2$) displays the pullback attractor corresponding to the critical rate $r=r_c$, the rate for which the pullback attractor forms a heteroclinic connection between the base state at the past asymptotic location and the edge state at the future asymptotic location of the habitat.

\subsection{Critical rate and heteroclinic connection}\label{sec:transverse}

To demonstrate that there is a unique critical rate that separates end-point tracking from extinction, we compute the pullback attractor corresponding to a representative set of rates, $r$. In order to bring out the dependence on $r$, we will use the notation $\mathcal{U}(r)=\{U(t,r)| -\infty<t<+\infty\}$. We determine if, for a given rate $r$, the pullback attractor corresponds to end-point tracking or extinction. This is accomplished by approximating the large time limit of the pullback attractor and then taking the norm of the difference of this with the stable pulse $u_2^*$, as well as with the trivial solution, $u_0^*$. 

Specifically, we compute the pullback attractor for a uniform grid of $N$ rates spanning the interval $[r_{\min}, r_{\max}]$, where the bounds are chosen based on preliminary simulations to ensure they bracket the critical rate. For each rate $r_i$, we evaluate the solution at $t = 1000$, which we have verified is sufficiently large for the pullback attractor to have approached its limiting behavior. We then compute two key metrics: the $L^2$ norm $\|U(1000)\|_2$, which measures the total population density, and the distance to the stable pulse $\|U(1000)-u^*_{2}\|_2$, which quantifies deviation from the tracking state.

The classification into end-point tracking versus extinction regimes is based on threshold criteria. If $\|U(1000)-u^*_{2}\|_2 < \delta_1$ for some small tolerance $\delta_1$, we classify the behavior as end-point tracking. Conversely, if $\|U(1000)\|_2 < \delta_2$ for some small tolerance $\delta_2$, we classify the behavior as extinction. In practice, we observe a clear separation between these two regimes with no ambiguous intermediate cases, provided the simulation time is sufficiently long.

We see in Figure \ref{fig:habitatfig118} a single critical rate $r$, corresponding to a jump in the colors, that separates end-point tracking behavior from extinction. The black curve, representing $\|U(1000)\|_2$, exhibits a transition from values near $\|u^*_2\|_2$ to values near zero, while the blue curve, representing $\|U(1000)-u^*_2\|_2$, shows the complementary transition from near-zero values to values comparable to $\|u^*_2\|_2$. The abruptness of this transition and the absence of multiple such transitions support our conjecture regarding the uniqueness of the critical rate. This numerical evidence strongly suggests that the system undergoes tipping at a single, well-defined critical rate, rather than exhibiting multiple critical rates.

\begin{figure}[htbp]
 \begin{center}
$
\begin{array}{lcr}
\includegraphics[scale=0.5]{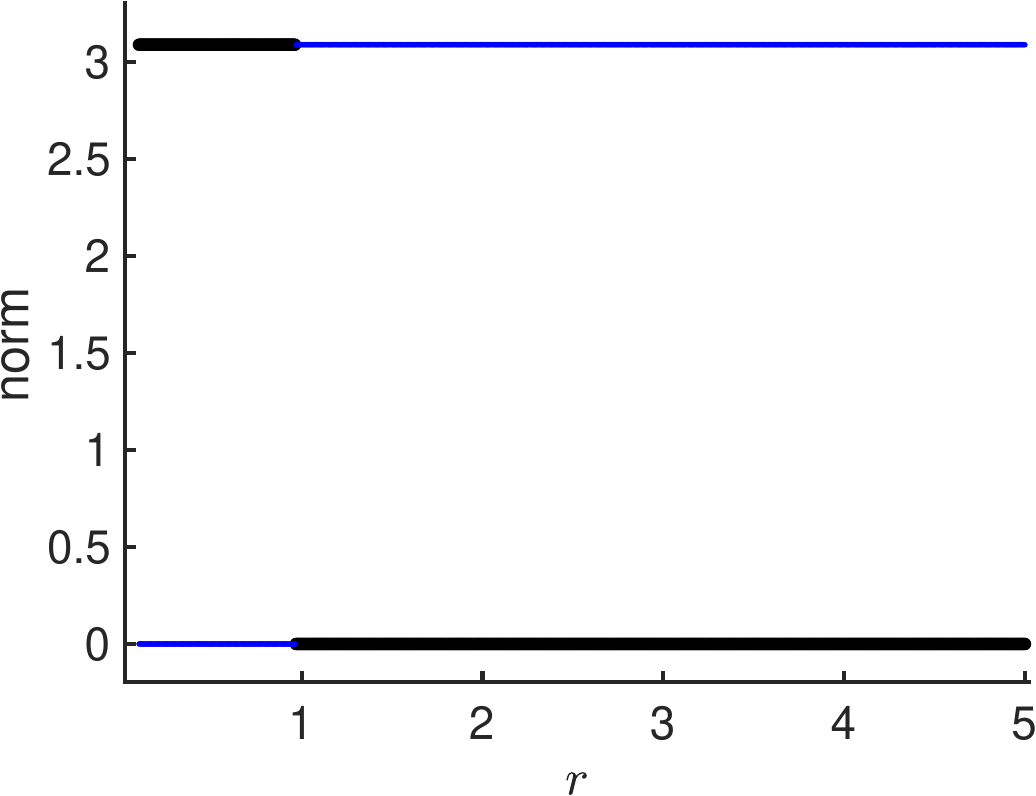}
\end{array}
$
\end{center}
\caption{With $(U(t,r),\gamma (t))$ being the solution to \eqref{main_mov_copy}, the black curve plots an approximation of $\|U(1000)\|_2$ and the blue curve of  $\|U(1000)-u^*_{2}\|_2$, where $u^*_{2}$ is the stable pulse evaluated at the spatial nodes. 
}
\label{fig:habitatfig118}
\end{figure}

At the critical rate, the pullback attractor forms a heteroclinic connection between the base state at the past asymptotic location and the edge state at the future asymptotic location of the habitat.  We compute this heteroclinic connection as a boundary/interior value  problem (BIVP). Knowing an approximate value of the critical rate is a key starting point since that is the parameter value at which the heteroclinic orbit exists. Finding the heteroclinic connection numerically is an important mathematical challenge, given the inifinite-dimensionality of the phase space. But it also has physical significance as it affords a more precise determination of the critical rate.

 We first approximate the critical rate using a bisection method. That is, we verify that the pullback attractor, computed as described in Section \ref{subsec:pullback}, for a rate $r_1$ corresponds to tracking and for a rate $r_2$  corresponds to extinction. We then compute $r_3:= (r_1+r_2)/2$ and determine whether the pullback attractor for this rate corresponds to tracking or extinction. We then replace $r_1$ or $r_2$ with $r_3$ so that again we have two rates, one corresponding to tracking and one to extinction. The pattern continues until the two rates, $r_1$ and $r_2$, are very close together. 

  We then take $r_1$ as a starting approximation of the critical rate, and the corresponding pullback attractor to approximate the heteroclinic connection. To approximate the critical rate $r_c$ with greater precision, and the associated heteroclinic connection, we form a boundary value type problem, which will also include a condition on the solution in the interior of the domain interval. Specifically, we  solve \eqref{eq:compactified_ODE} as a BIVP with conditions that correspond to the solution lying on the unstable manifold of the fixed point $(V_1^s,...,V_N^s,-a)$ at $t = -\infty$, and lying on the stable manifold of the fixed point $(V_1^u,...,V_N^u,a)$, corresponding to the unstable pulse, at $t = +\infty$.

The boundary and interior conditions are given as follows. Let $P_{-\infty}^s$ be the projection onto the stable subspace of $J(V_1^s,...,V_N^s,-a)$ and let $P_{+\infty}^u$ be the projection onto the unstable subspace of $J(V_1^u,...,V_N^u,a)$. For details on how to compute these subspaces numerically, see \cite{TrefethenBau1997}. Then the boundary conditions and interior condition are
\eq{
0=P_{+\infty}^u\left(\begin{bmatrix} V_1(+\infty)\\\vdots\\V_N(+\infty)\\\gamma(+\infty) \end{bmatrix}-\begin{bmatrix} V_1^u\\\vdots\\V_N^u\\a\end{bmatrix}\right),\quad
0=P_{-\infty}^s\left(\begin{bmatrix} V_1(-\infty)\\\vdots\\V_N(-\infty)\\\gamma(-\infty) \end{bmatrix}-\begin{bmatrix} V_1^s\\\vdots\\V_N^s\\-a\end{bmatrix}\right),\quad
0&= \gamma(0).
}{\notag}
We solve this BIVP on the same interval as in the bisection method, and use the solution from the bisection method as an initial guess. A plot of the heteroclinic as solved for by this method is shown in Figure \ref{fig:heteroclinic}. The precise critical rate, found by this method, gives us a precise picture of when R-tipping occurs in this canonical example. 
This critical rate is plotted as a function of $d(=2a)$ versus $r$ in Figure~\ref{fig:criticalvalueplot}.

As discussed in Section \ref{main_thm}, the heteroclinic connection is an intersection of $W^u(u_2^*,-a)$, which is the lift of the pullback attractor $\mathcal{U}$ to the autonomous system \eqref{main_mov}, with $W^s(u_1^*,+a)$, at $r=r_c$, and it is natural to ask whether this intersection is transverse. It is also of interest to see that the manifold crosses in the expected direction, which corresponds to survival to $r<r_c$ and extinction when $r>r_c$. The transversality shows that this division is precise, and the extinction/survival distinction holds arbitrarily close to $r_c$. 

To show transversality, we begin by augmenting system \eqref{main_mov} with the parameter $r$, 
\begin{equation}\label{main_mov_aug}
\begin{split}
u_t=& u_{zz}+r\,g(\gamma)\,u_z+f\left(u,H(z)\right),\\
   \gamma_t=& r\,g(\gamma),\\r_t=&0.
\end{split}
\end{equation}
The linearization of \eqref{main_mov_aug} about the heteroclinic solution $(\bar u, \gamma_c,r_c)$ is
\eq{
u_t&= u_{zz}+r(a-\gamma_c^2/a)\bar{u}_z -2r_c\gamma_c\ \bar{u}_z\gamma/a+r_c(a-\gamma_c^2/a)u_z+\frac{df}{du}(\bar u,H(z))u\\
\gamma_t &= -2r_c\gamma_c \gamma/a+(a-\gamma_c^2/a)r\\
r_t&= 0.
}{\label{eq:jac_expanded}}

\begin{figure}[t] 
 \begin{center}
$
\begin{array}{lcr}
 \includegraphics[scale=0.4]{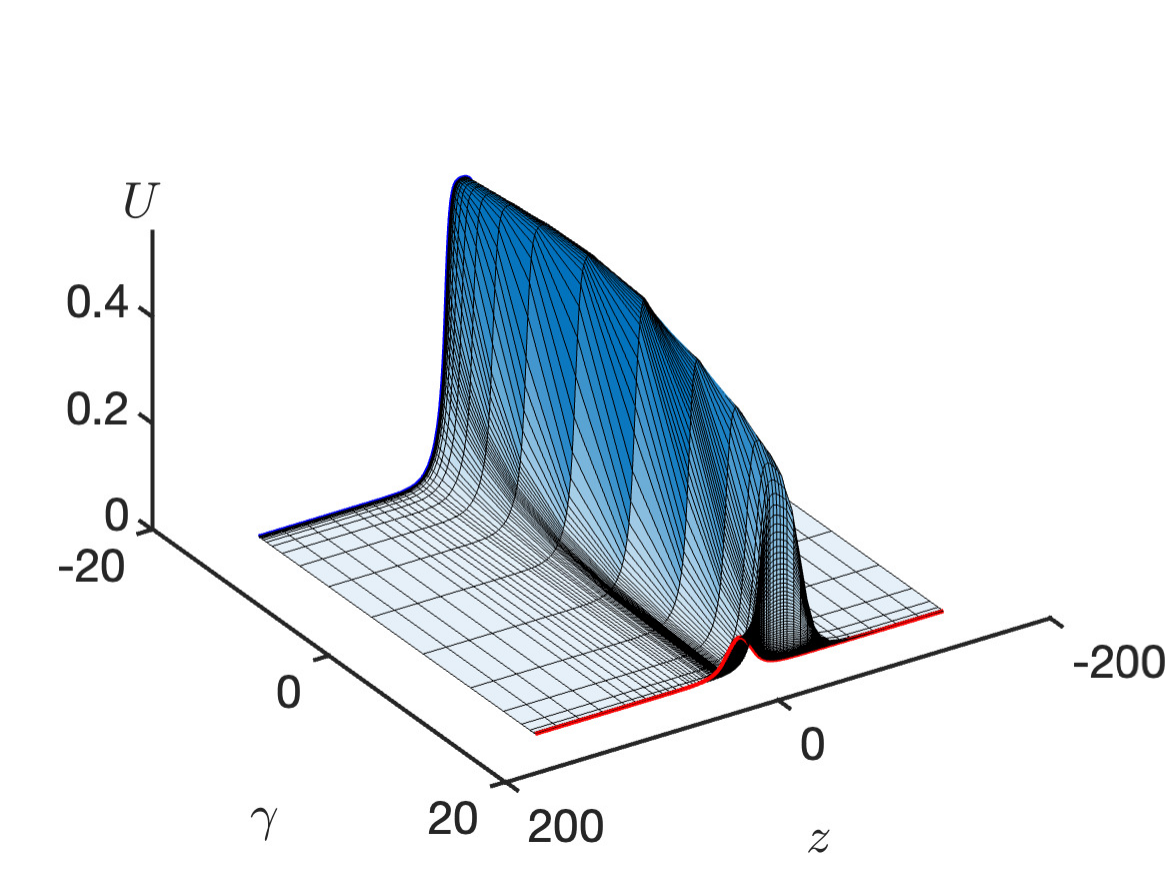}& \includegraphics[scale=0.4]{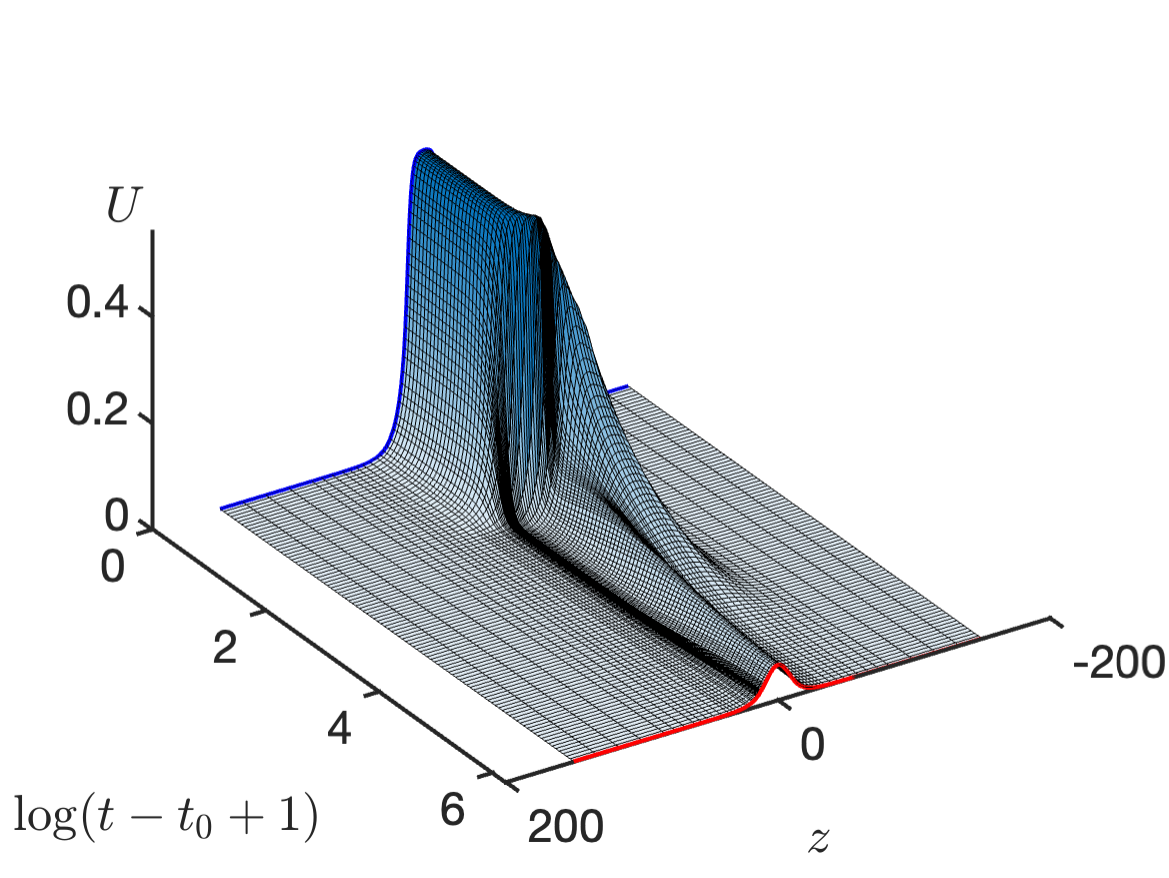} 
\end{array}
$
\end{center}
\caption{\text
(Left column) Plot, computing the heteroclinic connection with the boundary/interior value approach (BIVP), of $U$ against $z$ and $\gamma$ and (right column) plot of $U$ against  $z$ and $\log(t-t_0+1)$ for $r = r_c$ where $r_c\approx 0.9670$ is the critical value at which the system undergoes R-tipping. Other parameters are $a = 15.65$, $\beta = 0.15$, $\lambda = 4\beta$, and $L = 25$. 
}
\label{fig:heteroclinic}
\end{figure}

Using the linearized system, we calculate the tangent vector to 
\begin{equation}\label{eq:wculoc}\{\left(W^u(u_2^*,-a,r),r \right): r_c-\delta <r<r_c + \delta, 0<\delta \ll 1 \},\end{equation} 
which is normal to $W^u(u_2^*,-a,r_c)$. We then take the inner product of this tangent vector with a normal vector to the stable manifold $W^s(u_1^*,a,r)$ and show it is non-zero. 
Note that the manifold given by \eqref{eq:wculoc} can be viewed as the center-unstable manifold of the point $\left(u_2^*,-a,r_c\right)$, which shows it is a manifold, see \cite{batesjones}.
 
Turning \eqref{eq:jac_expanded} into an ODE system, and discretizing, we have
\eq{
V_t&= D_{zz}V+r(a-\gamma_c^2/a)D_z\bar V-2r_c\gamma_c \gamma D_z\bar V/a + r_c(a-\gamma_c^2/a)D_zV+\left[ \frac{df}{du} (\bar V_j,H(z_j))V_j\right]\\
\gamma_t&= -2r_c\gamma_c \gamma/a+(a-\gamma_c^2/a)r\\
r_t& = 0
}{\label{eq:expandedODE}}
The Jacobian of the RHS of \eqref{eq:expandedODE} is given by
\eq{
J(V,\gamma,r) = \mat{ D_{zz}+r_c(a-\gamma_c^2/a)D_z+\textrm{diag}\left(\left[ \frac{df}{du} (\bar V_j,H(z_j))\right]\right) & -\frac{2r_c\gamma_c }{a}D_z\bar V & (a-\gamma_c^2/a)D_z\bar V\\
0&-2r_c \gamma_c/a&a-\gamma_c^2/a\\
0&0&0}
}{\notag}
We note that at $t=+\infty$, $a-\gamma_c^2/a=0$. We solve the linearized system initialized with $(V,\gamma,r)|_{t=-T} = (0,0,1)^T$ and evaluate at $t=+T$ to obtain $(V(T),\gamma(T),r(T))$. We then normalize this vector to obtain a unit tangent vector  of the unstable manifold $W^u(u_2^*,-a,r)$ as $r$ varies. 

We use this setup to show that the unstable manifold $W^u(u_2^*,-a,r)$ intersects the stable manifold $W^s(u_1^*,a,r_c)$, as $r$ varies, transversally at $r=r_c$. The vector tangent to the unstable manifold $W^u(u_1^*,a,r_c)$ is normal to the stable manifold $W^s(u_1,a,r_c)$, since the eigenfunctions of the linearization are orthogonal due its self-adjointness. We approximate this normal vector by computing the eigenvector of $J(u_1^*,a,r_c)$ that corresponds to the growth eigenvalue. We ensure that the eigenvector corresponds to the direction of the flow from $u_1^*$ to $u_2^*$ by making the largest component of the vector positive through negation of the eigenvector if necessary. We also normalize this vector.

The inner product of the unit tangent vector to $W^{cu}(u_2^*,-a,r)$, normal to $W^{u}(u_2^*,-a,r_c)$ and the unit tangent vector to $W^u(u_1^*,a,r)$ at $r=r_c$ is non-zero, which confirms the transversallity of the intersection of $W^u(u_2^*,-a,r)$, as $r$ varies, with $W^s(u_1^*,a,r)$ at $r = r_c$.
 By the implicit function theorem, this establishes $r_c$ as a sharp threshold: the heteroclinic connection exists only at $r=r_c$, separating extinction ($r>r_c$) from survival ($r<r_c$). It is also seen to have the appropriate sign, i.e., consistent with the switch from survival to extinction as $r$ increases. Following the values taken above, when $L=25$ and $r=1$ ($a\approx 15.65$), the inner product is computed to be -0.0037. This shows that the tangent vector to the family (as $r$ varies) of the pullback attractors  $W^u(u_2^*,-a,r)$ points in the opposite direction to the normal vector, $W^u(u_1^*,a,r)$, which points in the direction of the flow that moves from near $u_1^*$ to $u_2^*$.

\section{Conclusion and discussions}\label{sec:conclusion}

We analyzed rate-induced tipping (R-tipping) phenomena in a reaction-diffusion model representing species dynamics under non-uniformly shifting habitat conditions. Our investigation has revealed several key insights into the dynamics of populations facing rapid environmental changes. More specifically, we have demonstrated numerically the existence of a critical rate $r_c(d)$ for a sufficiently large displacement $d>d^*$, at which the system undergoes a sudden transition. This finding underscores the importance of considering not just the magnitude of environmental changes, but also the speed at which they occur. The identification of this critical threshold provides a quantitative framework for assessing the vulnerability of populations to rapid habitat shifts.

In the limit of fast habitat shifts, we proved that R-tipping occurs for sufficiently large displacements. This result highlights the potential for abrupt population collapses when environmental changes outpace the species' ability to adapt or migrate. It serves as a cautionary note for conservation efforts, emphasizing the need for proactive measures in the face of accelerating environmental changes. Conversely, our analysis of the slow shift limit established that solutions can track the moving favorable habitat zone without tipping. This finding offers hope for species persistence under gradual environmental changes and suggests that slowing the rate of habitat alteration could be a vital strategy for conservation.

Our numerical simulations have revealed the  intricate structure of the critical transition and its dependence on system parameters. This insight provides a foundation for predicting how different species, characterized by various life history traits and environmental sensitivities, might respond to habitat shifts. The analytical techniques developed in this study, particularly our treatment of the fast and slow limits, extends the theoretical framework of dynamical systems, particularly in the context of true non-autonomous instabilities and rate-induced phenomena. These methods may find applications in analyzing rate-induced phenomena in other domains beyond ecology.

While our model is theoretical, the qualitative behaviors it predicts align with empirical observations of species responses to habitat changes \cite{pinsky2013marine,burrows2014geographical,schloss2012dispersal}. This concordance suggests that our framework could serve as a valuable tool for interpreting and predicting real-world ecological dynamics. The existence of critical rates and the geometry of R-tipping thresholds have significant implications for environmental policy and conservation strategies. It suggests that efforts to slow the pace of habitat alteration, even if the total change remains the same, could be crucial for preventing sudden biodiversity losses.

These results thus contribute to our understanding of how species might respond to rapid environmental changes, such as those induced by climate change. Some other possibilities for future work are listed below:
\begin{enumerate}
    \item Multi-species interactions: Extend the model to include multiple interacting species, investigating how competition or predation might affect R-tipping phenomena.
    \item Stochastic effects: Incorporate environmental noise or demographic stochasticity to study how random fluctuations influence tipping dynamics. Recent work \cite{slyman2023rate,slyman2025tipping} suggests that stochastic extensions could reveal important interactions between rate-induced and noise-induced tipping phenomena.
    \item Higher spatial dimensions: Extend the analysis to two or three spatial dimensions, which could reveal new patterns and phenomena not present in the one-dimensional case.
    \item Periodic forcing: Analyze the effects of seasonality or other periodic environmental variations on the tipping dynamics.
    \item Early warning signals: Develop and test statistical indicators that could provide early warnings of impending R-tipping in both models and real ecosystems.
\end{enumerate}

As we face unprecedented rates of global change, the ability to anticipate and potentially prevent such critical transitions becomes increasingly vital for maintaining biodiversity and ecosystem function. Our work lays a foundation for future research that could further bridge the gap between mathematical theory and ecological practice, ultimately contributing to more effective strategies for managing and conserving ecosystems in a rapidly changing world.

\section*{Acknowledgement}
B.B., E.F., M.H., and C.J. extend gratitude to the AIM workshop on {\it Computer-Assisted Proofs for Stability Analysis of Nonlinear Waves}, where they forged collaborative connections and began exploring ideas on projects related to traveling waves and related problems. E.F. was partially supported by NSF grant DMS-2137947.  M.H. was partially supported by the NSF grant DMS-2406623. C.J. was supported by the US Office of Naval Research Grant N00014-24-1-2198. 
SW was partially supported by Taighde Éireann—Research Ireland (Grant No. FFP-A/12066).

\bibliographystyle{abbrv}
\bibliography{refs}

@incollection {batesjones,
    AUTHOR = {Bates, Peter W. and Jones, Christopher K. R. T.},
     TITLE = {Invariant manifolds for semilinear partial differential
              equations},
 BOOKTITLE = {Dynamics reported, {V}ol.\ 2},
    SERIES = {Dynam. Report. Ser. Dynam. Systems Appl.},
    VOLUME = {2},
     PAGES = {1--38},
 PUBLISHER = {Wiley, Chichester},
      YEAR = {1989},
      ISBN = {0-471-91958-6},
   MRCLASS = {58D25 (34G20 35Q20 58F10)},
  MRNUMBER = {1000974},
MRREVIEWER = {John\ M.\ Ball},
}

@book{henry1981geometric,
  address = {Berlin},
  author = {Henry, Daniel B.},
  publisher = {Springer},
  series = {Lecture Notes in Mathematics},
  title = {{Geometric Theory of Semilinear Parabolic Equations}},
  volume = {840},
  year = {1981}
}

@book{TrefethenBau1997,
  author = {Trefethen, Lloyd N. and Bau, David},
  title = {Numerical Linear Algebra},
  publisher = {SIAM},
  year = {1997},
  isbn = {0898713617}
}

@article{berestycki2009can,
  title={Can a species keep pace with a shifting climate?},
  author={Berestycki, Henri and Diekmann, Odo and Nagelkerke, Cornelis J and Zegeling, Paul A},
  journal={Bulletin of mathematical biology},
  volume={71},
  pages={399--429},
  year={2009},
  publisher={Springer}
}

@book{pazy1983semigroups,
  title={Semigroups of Linear Operators and Applications to Partial Differential Equations},
  author={Pazy, Amnon},
  year={1983},
  publisher={Springer-Verlag},
  address={New York},
  series={Applied Mathematical Sciences},
  volume={44},
  isbn={978-0-387-90845-8}
}

@book{kloeden2011nonautonomous,
  title={Nonautonomous dynamical systems},
  author={Kloeden, Peter E and Rasmussen, Martin},
  number={176},
  year={2011},
  publisher={American Mathematical Soc.}
}

@article{ashwin2017parameter,
  title={Parameter shifts for nonautonomous systems in low dimension: bifurcation-and rate-induced tipping},
  author={Ashwin, Peter and Perryman, Clare and Wieczorek, Sebastian},
  journal={Nonlinearity},
  volume={30},
  number={6},
  pages={2185},
  year={2017},
  publisher={IOP Publishing}
}

@article{wieczorek2021,
  title={Compactification for asymptotically autonomous dynamical systems: theory, applications and invariant manifolds},
  author={Wieczorek, Sebastian and Xie, Chun and Jones, Chris KRT},
  journal={Nonlinearity},
  volume={34},
  number={5},
  pages={2970},
  year={2021},
  publisher={IOP Publishing}
}

@article{wieczorek2023rate,
  title={Rate-induced tipping: Thresholds, edge states and connecting orbits},
  author={Wieczorek, Sebastian and Xie, Chun and Ashwin, Peter},
  journal={Nonlinearity},
  volume={36},
  number={6},
  pages={3238},
  year={2023},
  publisher={IOP Publishing}
}

@article{hasan2023rate,
  title={Rate-induced tipping in heterogeneous reaction-diffusion systems: An invariant manifold framework and geographically shifting ecosystems},
  author={Hasan, Cris R and C{\'a}rthaigh, Ruaidhr{\'\i} Mac and Wieczorek, Sebastian},
  journal={SIAM Journal on Applied Dynamical Systems},
  volume={22},
  number={4},
  pages={2991--3024},
  year={2023},
  publisher={SIAM}
}

@article{ritchie2023rate,
  title={Rate-induced tipping in natural and human systems},
  author={Ritchie, Paul DL and Alkhayuon, Hassan and Cox, Peter M and Wieczorek, Sebastian},
  journal={Earth System Dynamics},
  volume={14},
  number={3},
  pages={669--683},
  year={2023},
  publisher={Copernicus GmbH}
}

@article{chen2011rapid,
  title={Rapid range shifts of species associated with high levels of climate warming},
  author={Chen, I-Ching and Hill, Jane K and Ohlem{\"u}ller, Ralf and Roy, David B and Thomas, Chris D},
  journal={Science},
  volume={333},
  number={6045},
  pages={1024--1026},
  year={2011},
  publisher={American Association for the Advancement of Science}
}

@article{parmesan2003globally,
  title={A globally coherent fingerprint of climate change impacts across natural systems},
  author={Parmesan, Camille and Yohe, Gary},
  journal={nature},
  volume={421},
  number={6918},
  pages={37--42},
  year={2003},
  publisher={Nature Publishing Group UK London}
}

@article{diffenbaugh2013changes,
  title={Changes in ecologically critical terrestrial climate conditions},
  author={Diffenbaugh, Noah S and Field, Christopher B},
  journal={Science},
  volume={341},
  number={6145},
  pages={486--492},
  year={2013},
  publisher={American Association for the Advancement of Science}
}

@article{loarie2009velocity,
  title={The velocity of climate change},
  author={Loarie, Scott R and Duffy, Philip B and Hamilton, Healy and Asner, Gregory P and Field, Christopher B and Ackerly, David D},
  journal={Nature},
  volume={462},
  number={7276},
  pages={1052--1055},
  year={2009},
  publisher={Nature Publishing Group UK London}
}

@article{pinsky2013marine,
  title={Marine taxa track local climate velocities},
  author={Pinsky, Malin L and Worm, Boris and Fogarty, Michael J and Sarmiento, Jorge L and Levin, Simon A},
  journal={Science},
  volume={341},
  number={6151},
  pages={1239--1242},
  year={2013},
  publisher={American Association for the Advancement of Science}
}

@article{burrows2014geographical,
  title={Geographical limits to species-range shifts are suggested by climate velocity},
  author={Burrows, Michael T and Schoeman, David S and Richardson, Anthony J and Molinos, Jorge Garc{\'\i}a and Hoffmann, Ary and Buckley, Lauren B and Moore, Pippa J and Brown, Christopher J and Bruno, John F and Duarte, Carlos M and others},
  journal={Nature},
  volume={507},
  number={7493},
  pages={492--495},
  year={2014},
  publisher={Nature Publishing Group UK London}
}

@article{schloss2012dispersal,
  title={Dispersal will limit ability of mammals to track climate change in the Western Hemisphere},
  author={Schloss, Carrie A and Nu{\~n}ez, Tristan A and Lawler, Joshua J},
  journal={Proceedings of the National Academy of Sciences},
  volume={109},
  number={22},
  pages={8606--8611},
  year={2012},
  publisher={National Academy of Sciences}
}

@article{slyman2025tipping,
  title={Tipping mechanisms in a carbon cycle model},
  author={Slyman, Katherine and Fleurantin, Emmanuel and Jones, Christopher KRT},
  journal={Chaos: An Interdisciplinary Journal of Nonlinear Science},
  volume={35},
  number={5},
  year={2025},
  publisher={AIP Publishing}
}

@article{slyman2023rate,
  title={Rate and noise-induced tipping working in concert},
  author={Slyman, Katherine and Jones, Christopher K},
  journal={Chaos: An Interdisciplinary Journal of Nonlinear Science},
  volume={33},
  number={1},
  year={2023},
  publisher={AIP Publishing}
}

\end{document}